\documentclass{article}
\usepackage[german, english]{babel}
\usepackage{amssymb,amsmath}
\usepackage{amsmath, amsthm}
\usepackage{amsfonts,epsfig}
\usepackage{hyperref}

\usepackage[margin=1.1in]{geometry}

\newtheoremstyle{theorem}
  {10pt}		  
  {10pt}  
  {\sl}  
  {\parindent}     
  {\bf}  
  {. }    
  { }    
  {}     
\theoremstyle{theorem}
\newtheorem{theorem}{Theorem}
\newtheorem{corollary}{Corollary}
\newtheorem{proposition}{Proposition}
\newtheorem{remark}{Remark}
\newtheorem{example}{Example}

\newtheoremstyle{defi}
  {10pt}		  
  {10pt}  
  {\rm}  
  {\parindent}     
  {\bf}  
  {. }    
  { }    
  {}     
\theoremstyle{defi}
\newtheorem{definition}{Definition}

\newcommand{\naturals}{\mathbb{N}}
\newcommand{\lengths}{\mathfrak{M}}
\newcommand{\clengths}{\mathfrak{M}'}

\newcommand{\complex}{\mathbb{C}}

\newcommand{\length}{\ell}
\newcommand{\nul}{{\rm nul}}
\newcommand{\rank}{{\rm rank}}
\newcommand{\defic}{{\rm def}}
\newcommand{\RS}{{\rm RS}}
\newcommand{\LNS}{{\rm LNS}}
\newcommand{\RNS}{{\rm RNS}}

\newcommand{\card}{{\rm card}}
\newcommand{\Span}{{\rm span}}
\newcommand{\RFM}{\rm \mathbb{R}\mathbb{F}\mathbb{M}_\omega({\complex})}

\newcommand{\Com}{{\complex}^\infty}
\newcommand{\QHF}{{\rm \mathbf{H}}}
\newcommand{\A}{{\rm \mathbf{A}}}
\newcommand{\B}{{\rm \mathbf{B}}}
\newcommand{\Q}{{\rm \mathbf{Q}}}
\newcommand{\G}{{\rm \mathbf{G}}}
\newcommand{\C}{{\rm \mathbf{C}}}
\newcommand{\e}{{\rm \mathbf{e}}}
\newcommand{\I}{{\rm \mathbf{I}}}
\newcommand{\R}{{\rm \mathbf{R}}}

\begin{document}
\bibliographystyle{plain}

\title{\textbf{\Large{The Solution of Row-Finite Linear Systems with the Infinite Gauss-Jordan Elimination}\\ 
\normalsize{The Case of Linear Difference Equations with Variable Coefficients}}
}
\author{Alexandros G. Paraskevopoulos\\
\small The Center for Research and Applications of Nonlinear Systems (CRANS)\\
\small Department of Mathematics, Division of Applied Analysis,\\
\small University of Patras, 26500 Patra, Greece}

\maketitle

\begin{abstract} The construction of the general solution sequence of row-finite linear systems is accomplished by implementing -ad infinitum- the Gauss-Jordan algorithm under a rightmost pivot elimination strategy. The algorithm generates a basis (finite or Schauder) of the homogeneous solution space for row-finite systems.
The infinite Gaussian elimination part of the algorithm solves linear difference equations with variable coefficients of regular order, including equations of constant order and of ascending order.
The general solution thus obtained can be expressed as a single Hessenbergian.\\


\end{abstract}

\section{Introduction} \label{intro-sec} 
The Gaussian elimination process, the most powerful and long-standing method for solving systems of linear equations since ancient China (200 B.C.) and up to the modern era (see ~\cite{Ch:Ha,Ho:Gr}), has been implemented under a leftmost pivot elimination strategy. 
The Gaussian combined with the Jordan elimination process and with row permutations provide the (upper) row-reduced echelon form (RREF) of a matrix. This is the standard row canonical form of finite matrices, meaning that every finite matrix has a unique row-equivalent RREF. The choice of the RREF as target matrix, drives the elimination process to the use of the leftmost entries as pivot elements, in all algorithm formulations available for finite matrices.
Even though the lower row-reduced echelon form (LRREF) -produced through the rightmost pivot elimination\footnote{Introduced in earlier work, see A. G. Paraskevopoulos,  The Infinite Gauss-Jordan Elimination on $\omega\times\omega$ Row-Finite Matrices~\cite{Pa:IGJ}.}- is also row canonical, the upper row-reduced echelon form (URREF) is the one unilaterally recognised as the identical to the RREF. But the URREF as target matrix prevents the application of the Gauss-Jordan elimination to row-finite systems, because a wide range of row-finite matrices are not left-associates (or row-equivalent) with row-finite matrices having zero entries above and below leftmost 1s. This is an essential postulate to matrices in URREF, (see example \ref{sec:Example1}).

Following earlier work of Toeplitz~\cite{To:Auf} (1909), Fulkerson introduced, in his doctoral thesis~\cite{Fu:Th} (1951) the notion of the ``quasi-Hermite form" (QHF) of a row-finite matrix, establishing the existence and ``almost uniqueness" of such matrix forms. Yet, this form has remained largely ignored for a long period\footnote{Fulkerson's thesis remained hidden in a university storage of Wisconsin since 1951. Thanks to Travis D. Warwick, librarian of the Stephen Cole Kleene Mathematics Library at University of Wisconsin - Madison, for making Fulkerson's Thesis available to me in 2009. It was then brought to light and now can be accessed \url{www.researchgate.net/publication/36218809}).}. The QHF is similar to the LRREF of a matrix, with the exception of all zero rows
being grouped together at the top or the bottom of the matrix (see section \ref{sec:QuasiHermiteForm}).

Fulkerson's proof on the existence of a row-equivalent QHF of a row-finite matrix $\A$, necessarily involves the axiom of countable choice (as implicitly stated in~\cite[Theorem 3.1]{Fu:Th}). It enabled him to overcome the lack of a certain rule for choosing a Hamel basis from the length-equivalent classes, say $({\mathcal{K}}_{\length_i})_{i\in\naturals}$, of the row space of $\A$. This basis consists of representatives $\xi_i$s chosen from each ${\mathcal{K}}_{\length_i}$. The $\xi_i$s were used by Fulkerson to formulate a recurrence, yielding the set of non-zero rows of $\QHF$. This set is a Hermite basis of the row space of $\A$. Due to the non-constructive nature of the axiom of countable choice, its emergence prior to the recurrence reduces the constructibility of the latter to cases in which a basis of $\xi_i$s is given in advance by $\A$. As shown in this paper, the Gaussian elimination part of the infinite Gauss-Jordan algorithm induces a rule for choosing a Hamel basis from $({\mathcal{K}}_{\length_i})_{i\in\naturals}$. The algorithm in its full extent establishes the constructiveness of the QHF of row-finite matrices.

In this paper, we deal with infinite linear systems
\begin{equation} \label{MQLR} \A\cdot y=g, \end{equation}
whose coefficient matrix $\A$ is $\naturals\times\naturals$ row-finite (the number of nonzero entries in each row is finite).

The infinite Gauss-Jordan elimination algorithm follows the basic processes of the standard Gauss-Jordan algorithm with the following innovations:\vspace{-0.05in}
\begin{itemize}
\item It is repeatedly applied to an infinite sequence of finite matrices associated with $\A$.\vspace{-0.08in}
\item It is implemented under a rightmost pivot elimination strategy.
\end{itemize}
This infinite elimination process constructs a row-finite matrix $\QHF$, in QHF (see section \ref{sec:TheInfiniteGauss-JordanAlgorithmAndTheQuasi-HermiteForm}) by means of an infinite sequence of row-elementary operations reducing $\A$ to $\QHF$. The composite of this sequence is represented by a nonsingular row-finite matrix $\Q$, which ensures that $\A$ and $\QHF$ are left associates. Subsequently, left association guarantees that $\A$ and $\QHF$ have identical homogeneous solution spaces.

The general solution sequence of Eq. (\ref{MQLR}) is expressed in terms of entries of $\QHF$ and $\Q$ (see section \ref{sec:ConstructionOfRow-FiniteSystemSolutions}). 
If $\A$ possesses infinite deficiency, the infinite Gauss-Jordan algorithm generates a Schauder basis of the homogeneous solution space of Eq. (\ref{MQLR}), considered as subspace of the Fr\'{e}chet Space of complex sequences. Otherwise, the algorithm generates a finite basis (Hamel basis). Both types of bases are described in terms of opposite-sign columns of $\QHF$ not containing leading 1s and they generalize the notion of the \textit{fundamental solution set} (see section \ref{sec:AFundamentalSetOfHomogeneousSolutions}), primarily associated with linear difference equations (see ~\cite{El:Int}).

If the coefficient matrix $\A$ of Eq. (\ref{MQLR}) possesses a non-zero uppermost diagonal, then Eq. (\ref{MQLR}) represents a \textit{linear difference equation with variable coefficients} (LDEVC). Furthermore, if all entries of the uppermost diagonal are non-zero, $\A$ represents an equation of \emph{regular order} (RO-LDEVC). Otherwise, $\A$ represents an equation of \emph{irregular order}. Equations of regular order comprise linear recurrences with variable coefficients of ascending order (see subsection \ref{sec:TheNonhogeneousSolution}) and equations of $N$-order (see subsection \ref{sec:NOrderLinearDifferenceEquationWithVariableCoefficients}).

As the coefficient matrix $\A$ associated with a RO-LDEVC is in lower echelon form the infinite Gaussian elimination is sufficient for constructing the QHF, $\QHF$, of $\A$. Moreover as $\A$ is of finite (zero) left-nullity, $\QHF$ is the unique Hermite form (or LRREF) of $\A$. It is shown hereby, that the opposite-sign first $N$ columns of $\QHF$, constructed simultaneously by the infinite Gaussian elimination, yield $N$ linearly independent homogeneous solutions, thus forming a fundamental solution set.

The lack of a general method for solving linear difference equations of order greater than one with non constant  coefficients, has been recorded in mathematical literature  in~\cite{Goldberg} and more recently in~\cite{GAg,Jensen}. 

Also recorded is an alternative approach suggesting priority should be given to a generalization of the closed form solution of the first order linear recurrence with variable coefficients following the \textit{``guess-and-prove by induction"} approach (see ~\cite{El:Int}).

The infinite Gaussian elimination method bridges these two research lines. 
Implemented with symbolic computation (see section \ref{sec:ExplicitSolutionsForLinearDifferenceEquationsOfRegularOrder}), it results in fundamental solution sequences whose terms are determinant expansions of lower Hessenberg matrices (Hessenbergians). Subsequently, the solution formulas, expressed in terms of Hessenbergians, are generalized, through mathematical induction, to cover all RO-LDEVCs.
The general solution can also be expressed  as a single Hessenbergian. This turns out to be a natural generalization of both the standard solution of the first order and the solution of the $N$th order LDEVC, as established by Kittappa in ~\cite{Kit:Rep}.

Two examples of LDEVCs of irregular order are presented in section \ref{sec:LinearDifferenceEquationsOfIrregularOrder}\hspace{0.02in} in order to illustrate the infinite Gauss-Jordan elimination algorithm in its full implementation.

The present work opens new perspectives for the solution of LDECVs and their applications to time varying models. The Leibniz determinant formula applied on  Hessenbergians can be expressed as a sum of signed elementary products indexed by integers, instead of permutations (see~\cite{Pa:ClFS}). This, answers a long-standing open question concerning a closed form expression of the general solution for RO-LDEVCs.
A unified theory for time series models with variable coefficients is presented\footnote{This paper has been presented in several academic seminars (see for example presentation titled ``A unified theory for time varying models: foundations and applications in the presence of breaks and heteroskedasticity", October 2013, Birkbeck College,  University of London, on \url{http://www.ems.bbk.ac.uk/research/Seminar_info/unifiedtheory}).} in~\cite{Pa:UTH} and its effects on modelling returns during financial crises in~\cite{Ka:MRV}.


\section{The Infinite Gauss-Jordan Algorithm and the Quasi-Hermite Form of Row-finite Matrices}
\label{sec:TheInfiniteGauss-JordanAlgorithmAndTheQuasi-HermiteForm}
Throughout the paper $\omega$ denotes the least infinite ordinal number, which is identified with the standard set of natural numbers $\naturals$ including zero, $\complex$ stands for the algebraic field of complex numbers and $\complex^{(\omega)}$ stands for the linear space of infinite numerical sequences with finite number of non-zero terms in $\complex$. The set $({\complex^{(\omega)}})^\omega$ consists of the sequences in $\complex^{(\omega)}$. The set  $({\complex^{(\omega)}})^\omega$ equipped with the matrix addition and multiplication by scalars turns into the linear space of row-finite $\omega\times\omega$ matrices over $\complex$. If $({\complex^{(\omega)}})^\omega$ is further equipped with the matrix multiplication, linear and ring structures turn $({\complex^{(\omega)}})^\omega$ into an associative and non commutative algebra with the identity $\omega\times\omega$ matrix $\I$ as unit element. This algebra will be denoted by $\RFM$.
 
A matrix $\A\in\RFM$ will be also written as $\A=(a_{ij})_{(i,j)\in\omega\times\omega}$. If $\A$ is considered as a right operator, it induces the endomorphism: $\complex^{(\omega)}\ni x\mapsto x\cdot \A\in \complex^{(\omega)}$. If $\A$ is considered as a left operator, it induces the endomorphism: $\Com \ni x\mapsto  \A\cdot x\in \Com$. The coefficient matrix $\A\in \RFM$ in Eq. (\ref{MQLR}) is considered as a left operator on $\Com$. 
The row of index $i$ of $\A$ will be denoted by $\A_i\in \complex^{(\omega)}$.

Following Toeplitz~\cite{To:Auf} the column index of the rightmost non-zero element of the row $\A_i$ is called the \emph{length} of $\A_i$ denoted by $\length(\A_{i})$ or simply by $\length_i$. Notice that $\length(1,0,0,...)=0$. We adhere to the conventions: $\textbf{0}=(0,0,0,...)$ and $\length(\textbf{0})=-1$.  The \textit{canonical basis} of $\complex^{(\omega)}$ consisting of the sequences $\e_0=(1,0,0,...), \e_1=(0,1,0,...),...$, is denoted as $(\e_n)_{n\in\omega}$. If $\A_i\not= \textbf{0}$ and $\A_j\not=\textbf{0}$ such that $\length(\A_{i})<\length(\A_{j})$, whenever $i<j$, then $\A$ is said to be in \emph{lower row echelon form}.  The \emph{left-null space} of $\A$ is denoted by $\LNS(\A)=\{x\in \complex^{(\omega)}: x\cdot \A=\textbf{0}\}$. The \textit{row space} of $\A$ is the span of the rows of $\A$: $\RS(\A)=\Span(\A_i)_{i\in \omega}$. The \textit{left-nullity} and the \textit{rank} of $\A$ are the dimensions of $\LNS(\A)$ and $\RS(\A)$, respectively: $\nul(\A)=\dim \ \LNS(\A)$ and $\rank(\A)=\dim\ \RS(\A)$. 
\begin{definition} \label{def left association} $\A,\B\in \RFM$ 
are said to be: 
\begin{enumerate}
\item \emph{Left associates} if there exists a nonsingular matrix $\Q\in \RFM$ such that
\begin{equation} \label{left association}
                                                 \Q\cdot \A=\B. 
\end{equation}
\item \emph{Row equivalent} if $\RS(\A)=\RS(\B)$ and $\nul(\A)=\nul(\B)$. 
\end{enumerate}
\end{definition}
If $\A,\B$ are finite $n\times m$ matrices the condition $\RS(\A)=\RS(\B)$ is equivalent to the definition of row equivalence, while $\nul(\A)=\nul(\B)$ necessarily follows. However, if $\A,\B$ are row-finite $\omega\times\omega$ matrices, the condition $\nul(\A)=\nul(\B)$ is essential in the definition of row equivalence. For example consider the row-finite $\omega\times\omega$ matrices:
\[\A_1=\left(\begin{array}{ccccc}
1 & 0 & 0 &  0      & ...  \\
0 &  1  & 0 & 0     & ...  \\
0 &  0  & 1   & 0   & ... \vspace{-0.05in}\\
\vdots & \vdots   & \vdots   & \vdots      & \vdots\vdots\vdots  
\end{array} \right), \hspace{0.5in}
\A_2=\left(\begin{array}{ccccc}
0 & 0 & 0 &  0      & ...  \\
1 &  0  & 0 & 0     & ...  \\
0 &  1  & 0   & 0   & ...\\
0 &  0  & 1   & 0   & ... \vspace{-0.05in}\\
\vdots & \vdots   & \vdots   & \vdots      & \vdots\vdots\vdots  
\end{array} \right)
.\]
Evidently $\RS(\A_1)=\RS(\A_2)$  but $\nul(\A_1)\not=\nul(\A_2)$.
In view of definition \ref{def left association}, the following theorem generalizes a well known result for finite matrices.
\begin{theorem} \label{theorem of row equivalence}
$\A,\B\in \RFM$ are left associates if and only if $\A,\B$ are row-equivalent \emph{(see ~\cite{Fu:Th} for proof)}.
\end{theorem}
By virtue of theorem \ref{theorem of row equivalence}, the notions of row equivalence and left-association will be equivalently used throughout the paper.
\subsection{Quasi-Hermite Form of Row-Finite Matrices}
\label{sec:QuasiHermiteForm}
The following definition and the main results of this subsection have been established by Fulkerson in~ ~\cite{Fu:Th}.
\begin{definition} \label{definition-Hermite basis} 
A matrix $\QHF=(h_{ij})$ in $\RFM$ is in \emph{quasi-Hermite Form} (QHF) if the following properties hold:
\begin{description}
\item[\textit{i})] The sequence of lengths $(\length(\QHF_j))_{j\in J}$ of the non-zero rows $(\QHF_j)_{j\in J}$ of $\QHF$ is strictly increasing ($\QHF$ is in lower row-echelon form). 
\item[\textit{ii})]  $h_{j\length_{j}}=1$ for all $j\in J$ (The rightmost coefficients of non-zero rows are 1s).
\item[\textit{iii})]  If $m\in \omega$ and $j\in J$ such that $j\not=m$, then $h_{m \length_{j}}=0$ ($\QHF$ is in row-reduced form: the entries above and below rightmost 1s ($h_{j\length_j}=1$)  are all zero).
\end{description}
\end{definition}
Notice that the postulate (\textit{iii}) in the above definition differs form Fulkerson's third postulate in the definition of the QHF introduced along with the notion of Hermite basis.  This states that $h_{m \length_j}=0$ whenever $m>j$, that is all entries ($h_{m \length_j}$) below rightmost 1s ($h_{j\length_j}$) are zero. However, by virtue of statement (\textit{i}) of definition \ref{definition-Hermite basis}, all zero entries above rightmost 1s are also zero and thus the two definitions coincide. 

Under the assumption of the countable axiom of choice, every matrix $\A\in\RFM$ has a QHF. 

The indexing set of the zero rows of $\QHF$, will be denoted by $W$ and the indexing set of non-zero rows of $\QHF$ by $J=\omega\setminus W$. Dealing with row-finite matrices of infinite rank, the indexing set $J$ is an infinite subset of $\omega$ and we shall write for it $J=\{j_0,j_1,j_2,...\}$ assuming that $j_0<j_1<j_2<...$. The sequence of non-zero rows of $\QHF$ will be written as $(\QHF_{j_i})_{i\in\omega}$, whereas $j_i\in J$. The strictly increasing sequence of nonnegative row lengths of $\QHF$ is denoted by $(\mu_i)_{i\in\omega}$ with $\mu_i=\length(\QHF_{j_i})$ for $i\in \omega$. A typical row-finite matrix in QHF is exhibited below:
\begin{equation} \label{QHF}
\QHF\!\!=\!\!\!\left(\!\!\!\begin{array}{ccccccccccccccc} 
h_{j_00} &\!\!\!\! ... \!\!\!\!\! & h_{j_0\mu_0-1}  &\!\!\!\!  1 \!\!\!\! &     0            &\!\!\!\!  ... \!\!\!\! &           0     &\!\! 0 \!\!&     0           &\!\!\!\!  ... \!\!\!\! & 0 &\!\!\!\! 0 &\!\!\!\! 0 &\!\!\!\!  ...   \\ 
0      & \!\!\!\! ... \!\!\!\!\! &            0     &\!\!\!\!  0 \!\!\!\! &     0            &\!\!\!\!  ... \!\!\!\!&            0     &\!\! 0 \!\!&     0           &\!\!\!\!  ... \!\!\!\! &0 &\!\!\!\! 0 &\!\!\!\! 0 &\!\!\!\!  ...   \vspace{-0.03in}\\
h_{j_10} &\!\!\!\!  ... \!\!\!\!\!  &  h_{j_1\mu_0-1} &\!\!\!\!  0 \!\!\!\! & h_{j_1\mu_0+1} &\!\!\!\! ... \!\!\!\! & h_{j_1 \mu_1-1}&\!\! 1 \!\!&     0           & \!\!\!\!  ... \!\!\!\! &0 &\!\!\!\! 0 &\!\!\!\! 0 &\!\!\!\!  ...  \\
 0      &\!\!\!\!  ... \!\!\!\!\!  & 0                 &\!\!\!\!  0 \!\!\!\! &   0              &\!\!\!\! ... \!\!\!\!&            0     &\!\! 0 \!\!&     0           & \!\!\!\!  ... \!\!\!\! &0&\!\!\!\! 0 &\!\!\!\! 0 &\!\!\!\!  ... \vspace{-0.03in} \\
h_{j_20} & \!\!\!\! ... \!\!\!\!\!  & h_{j_2\mu_0-1}  &\!\!\!\!  0 \!\!\!\! & h_{j_2\mu_0+1} &\!\!\!\! ... \!\!\!\!&        h_{j_2\mu_1-1}    &\!\! 0 \!\!& h_{j_2\mu_1+1}& \!\!\!\!  ... \!\!\!\! &h_{j_2\mu_2-1} &\!\!\!\! 1 &\!\!\!\! 0 &\!\!\!\!  ... \vspace{-0.01in}\\
0      &\!\!\!\!  ... \!\!\!\!\!  & 0                 &\!\!\!\!  0 \!\!\!\! &   0              &\!\!\!\! ... \!\!\!\!&            0     &\!\! 0 \!\!&     0           & \!\!\!\!  ... \!\!\!\! &0&\!\!\!\! 0 &\!\!\!\! 0 &\!\!\!\!  ... \vspace{-0.1in}\\
.      &\!\!\!\!  ... \!\!\!\!\! &       .           &\!\!\!\!  . \!\!\!\! &   .              &\!\!\!\!  ... \!\!\!\!&            .     &\!\! . \!\!&     .           & \!\!\!\!  ... \!\!\!\! &. &\!\!\!\! . &\!\!\!\! . &\!\!\!\!  ...  \vspace{-0.1in}\\
.      &\!\!\!\!  ... \!\!\!\!\! &       .           &\!\!\!\!  . \!\!\!\! &   .              &\!\!\!\!  ... \!\!\!\!&            .     &\!\! . \!\!&     .           & \!\!\!\!  ... \!\!\!\! &. &\!\!\!\! . &\!\!\!\! . &\!\!\!\!  ... \vspace{-0.1in}\\
.      &\!\!\!\!  ... \!\!\!\!\! &       .           &\!\!\!\!  . \!\!\!\! &   .              &\!\!\!\!  ... \!\!\!\!&            .     &\!\! . \!\!&     .           & \!\!\!\!  ... \!\!\!\! &. &\!\!\!\! . &\!\!\!\! . &\!\!\!\!  ... \end{array}\!\!\! \right). \end{equation}
The ``quasi-uniqueness" of a QHF of a row-finite matrix $\A$ is the result of the statement:\emph{``All \emph{QHFs} of $\A$ differ by row permutations"}.

If $\nul(\A)$ and $\rank(\A)$ are infinite and all the zero rows are positioned at the top (or the bottom part) of $\QHF$, then the row indexing set of $\QHF$ must be the ordinal $\omega+\omega=\omega2$. However, the inequality $\omega2>\omega$ leads to matrices not belonging in $\RFM$, a fact that undermines left association. 
In this case, we are obliged to position a finite number of consecutive zero rows between non-zero rows, thus ensuring that the indexing set of the rows of $\QHF$ is $\omega$. 
Following Fulkerson, we trade in the full uniqueness of the reduced matrix in order to gain results in $\RFM$.

If, on the other hand, $\nul(\A)$ is finite, then the full uniqueness of $\QHF$ is guaranteed by positioning zero rows at the top part of $\QHF$. Now the row indexing set is still $\omega$ and $\QHF$ is the HF (or LRREF) of $\A$.
\subsection{The Infinite Gauss-Jordan Elimination Algorithm}
\label{sec:TheInfiniteGauss-JordanEliminationAlgorithm}
The infinite Gauss-Jordan elimination algorithm is based on a routine alternative to the standard elimination algorithm for finite matrices, which can  furthermore be extended to row-finite matrices. This routine is repeatedly applied to a sequence of successively augmented finite matrices by means of rows of the original matrix. If the original matrix is finite, the algorithm, equipped either with a leftmost or with a rightmost pivot elimination strategy, provides a row-equivalent upper or lower row reduced matrices respectively. This is however not the case for row-finite matrices (see example \ref{sec:Example1}). In order to ensure the row equivalence of the original row-finite matrix to the target matrix, the latter is replaced with a QHF and the routine uses the rightmost non-zero entries of the rows,  as pivot elements.

The elimination part of the algorithm is supplemented with row permutations applied to the non-zero rows of each reduced matrix so as to construct a finite matrix in QHF. Meanwhile the zero rows either encountered or created are not permuted, ensuring that the row indexing set of the resulting row-finite matrix is the ordinal $\omega$. As shown in the following paragraphs the algorithm effectively reduces a matrix in $\RFM$ to a QHF in $\RFM$ preserving left association.

In all that follows, the \emph{greatest length} row among the first $n+1$ rows of $\A$ is denoted by $g\length(\A_i)_{0\le i\le n}$.  The $(n+1)\times(m+1)$ \emph{top submatrix} of $\A$ is defined as 
\[ \A^{(n)}=(a_{i,j})_{0\le i\le n, 0\le j\le m}, \] whenever $m=g\length(\A_i)_{0\le i\le n}$. In the sequel the notation $g\length(\A^{(n)})$ and $g\length(\A_i)_{0\le i\le n}$ will be equivalently used.

\paragraph{Algorithm Description.} The elimination process starts with the top submatrix $\B^{(n_1)}\stackrel{\rm def}{=}\A^{(n_1)}$ of $\A$ containing at least two non-zero rows. The matrix $\B^{(n_1)}$ is reduced to a QHF denoted by ${\QHF}^{(n_1)}=(h^{(n_1)}_{ik})_{i,k}$. New consecutive rows of $\A$ are inserted below the last row of ${\QHF}^{(n_1)}$, until the first non-zero row, say $\A_{n_2}=(a_{n_2,k})_k$, is encountered. 

Let $\mu_2=\max\{\length(\A_{n_2}), g\length({\QHF}^{(n_1)})\}$. 
Define the $(n_2+1)\times(\mu_2+1)$ matrix  $\B^{(n_2)}=(\beta_{ik})_{i,k}$ as follows:
\[\beta_{ik}=\left\{\begin{array}{ll} h^{(n_1)}_{ik}  & {\rm if} \ 0\le i\le n_1\ {\rm and}\ 0\le k\le g\length({\QHF}^{(n_1)}) \\
  a_{n_2k} & {\rm if}\ i=n_2\ {\rm and}\ 0\le k\le \length(\A_{n_2})\\
  0 & {\rm otherwise}
\end{array}\right.  \]
The algorithm is then applied to $\B^{(n_2)}$ yielding a QHF of $\B^{(n_2)}$, say ${\QHF}^{(n_2)}=(h^{(n_2)}_{ik})_{i,k}$\vspace{0.05in}. This process continues ad infinitum, thus generating a sequence $({\QHF}^{(k)})_{k\ge 0}$ of finite matrices in QHF.
 
The submatrix of the first rows of ${\QHF}^{(k)}$ including the row ${\QHF}^{(k)}_n$, where $n\le k$ will be denoted by ${\QHF}^{(k)}\!\!\mid_n$ (hence ${\QHF}^{(n)}\!\!\mid_n={\QHF}^{(n)}$).
\paragraph{Algorithm Anatomy.}
\label{sec:AlgorithmAnatomy}
For theoretical purposes, the matrices ${\mathcal{B}}^{(n)}$, ${\mathcal{H}}^{(n)}$, ${\mathcal{H}}^{(k)}\!\!\mid_n$ are introduced by augmenting the matrices $\B^{(n)}, \QHF^{(n)}, {\QHF}^{(k)}\!\!\mid_n$ with an infinite number of zero columns further to their right, respectively. The row and the column dimension of the augmented matrices are $(n+1)$ and $\omega$ respectively. The goal of this setting is to provide matrices whose rows are elements of $\complex^{(\omega)}$, while preserving their row-length:
\[g\length({\mathcal{B}}^{(n)})=g\length(\B^{(n)}),\
g\length({\mathcal{H}}^{(n)})=g\length(\QHF^{(n)}),\
g\length({\mathcal{H}}^{(k)}\!\!\mid_n)=g\length(\QHF^{(k)}\!\!\mid_n).
\]
The algorithm is not affected by this modification. In particular, the row  $\A_{k}$ is inserted below ${\mathcal{H}}^{(k-1)}$ yielding the matrix ${\mathcal{B}}^{(k)}=({\mathcal{H}}^{(k-1)}:\A_{k})$ with row dimension $k+1$ and column dimension $\omega$. The algorithm is applied to the matrix ${\mathcal{B}}^{(k)}$ reducing it to ${\mathcal{H}}^{(k)}$. It generates a sequence of matrices $({\mathcal{H}}^{(k)})_{k\ge 0}$ in QHF, such that $\QHF^{(n)}$ is a submatrix of ${\mathcal{H}}^{(n)}$ and we shall denote it by:
\[\QHF^{(n)}\sqsubset {\mathcal{H}}^{(n)}.\] 
\begin{description}
\item[I)\hspace{0.15in} Gaussian elimination.] Let ${\mathcal{H}}^{(k-1)}_i$ be the $i$-row of the matrix ${\mathcal{H}}^{(k-1)}$. As ${\mathcal{H}}^{(k-1)}$ is in QHF, the lengths of the non-zero rows of ${\mathcal{H}}^{(k-1)}$ form a strictly increasing sequence (lower echelon form).
Let $\A_k$ be the new inserted row. Let also $\A_k\not=\textbf{0}$. The algorithm uses the rightmost non-zero elements of the rows of ${\mathcal{H}}^{(k-1)}$ as pivots to clear entries of $\A_k$. The Gaussian elimination will change $\A_k$ if and only if there is a non-zero row of ${\mathcal{H}}^{(k-1)}$, say ${\mathcal{H}}^{(k-1)}_{m}$, $0\le m\le k-1$, such that $\length({\mathcal{H}}^{(k-1)}_{m})\le\length(\A_k)$ and $a_{k,j}\not=0$ with $j=\length({\mathcal{H}}^{(k-1)}_{m})$. The $\length(\A_k)$ will change if and only if there is a non-zero row, say ${\mathcal{H}}^{(k-1)}_{m}$, $0\le m\le k-1$, such that $\length({\mathcal{H}}^{(k-1)}_{m})=\length(\A_k)$\vspace{0.02in}.
Upon completion of the Gaussian elimination and normalization of the rightmost coefficients to $1$, the resulting $k$-row is denoted by $\G_k$. The row $\G_k$ is a linear combination of the inserted non-zero rows of $\A$ that is 
		\begin{equation}\label{G is LC of A}
     \G_k=\sum_{i=0}^nc_{ki}\A_i \ \ {\rm  with} \ \ c_{kk}\not=0.	
		\end{equation}
In case $\A_k=\textbf{0}$, then $\G_k=\textbf{0}$, which is also covered by Eq. (\ref{G is LC of A}), by taking $c_{kk}=1$ and $c_{ki}=0$ for $0\le i\le k-1$. Consequently, if $\G_k\not=\textbf{0}$, then $\length(\G_k)\not=\length({\mathcal{H}}^{(k-1)}_i)$ for all $i: 0\le i \le k-1$. In other words $\length(\G_k)$ differs from the lengths of all rows of  ${\mathcal{H}}^{(k-1)}$. Accordingly, for every $k\in\naturals$ such that $\G_k\not=\textbf{0}$, we have:
\begin{equation}\label{length of Gk}
 \length(\G_k)\not=g\length({\mathcal{H}}^{(k-1)}).
\end{equation}
It turns out that there are the following possible cases:
\begin{description}
\item[Case \textit{i})]  $\length(\G_k)>g\length({\mathcal{H}}^{(k-1)})$ or $\G_k=\textbf{0}$. In this case $({\mathcal{H}}^{(k-1)}:\G_k)$ is in QHF and therefore  $({\mathcal{H}}^{(k-1)}:\G_k)={\mathcal{H}}^{(k)}$. The algorithm continues by inserting the new row $\A_{k+1}$. 
\item[Case \textit{ii})]  $\length(\G_k)<g\length({\mathcal{H}}^{(k-1)})$ and $\G_k\not=\textbf{0}$. In this case the algorithm goes to the next elimination step \textbf{(II)}. 
\end{description}
\item[II)\hspace{0.1in} Jordan elimination.] The row $\G_k$ is used as a pivot row to eliminate the non-zero entries $(i,\length(\G_k))$, $0\le i\le k-1$, of the rows in ${\mathcal{H}}^{(k-1)}$, namely the entries above the leading 1. Then the algorithm goes to \textbf{(III)}.
\item[III)\hspace{0.025in} Row permutations.] Upon completion of the Jordan elimination row permutations take place. If $j_{i-1},j_i\in J$ ( $0\le j_{i-1}< j_i<k$) and $\length({\mathcal{H}}^{(k-1)}_{j_{i-1}})<\length(\G_k)<\length({\mathcal{H}}^{(k-1)}_{j_i})$ or $\length(\G_k)<\length({\mathcal{H}}^{(k-1)}_{j_0})$ then $\G_k$ takes the position of  ${\mathcal{H}}^{(k-1)}_{j_i}$ or ${\mathcal{H}}^{(k-1)}_{j_0}$ respectively. Now $\G_k$ is designated by ${\mathcal{H}}^{(k)}_{j_i}$. In the meantime all the non-zero rows below and including the row ${\mathcal{H}}^{(k-1)}_{j_i}$ move downwards, preserving the lower echelon form of the new matrix ${\mathcal{H}}^{(k)}$, while leaving unchanged the position of zero rows. 
Thus, if $0\le n\le k-1$, then $\length({\mathcal{H}}^{(k)}_n)\le \length({\mathcal{H}}^{(k-1)}_n)$, whence $g\length({\mathcal{H}}^{(k)}|_n)\le g\length({\mathcal{H}}^{(k-1)}|_n)$.

More generally,
\begin{equation}\label{inequality1} 
g\length({\mathcal{H}}^{(k)}|_n)\le g\length({\mathcal{H}}^{(N)}|_n)
\end{equation} 
for $n\le N < k$. Row permutations on ${\mathcal{H}}^{(N)}|_n$ occur, if and only if
\begin{equation}\label{inequality2} 
g\length({\mathcal{H}}^{(k)}|_n)<g\length({\mathcal{H}}^{(N)}|_n),
\end{equation}
for some $k>N$.

Upon completion of row-permutations, the matrix ${\mathcal{H}}^{(k)}$ is in QHF and the process continues by inserting the new row $\A_{k+1}$ below it. 
\end{description}
\begin{remark}\label{remark1} 
{\rm A summary of results and direct extensions derived from the above algorithm, are presented below:
\begin{description}
\item[I)] The Jordan elimination does not affect the row-lengths of ${\mathcal{H}}^{(k-1)}$. New row-lengths and new zero rows are exclusively generated by the Gaussian elimination.
\item[II)] Eq. (\ref{length of Gk})  can be generalized as follows:
 
Let $N\in\naturals$. If $\G_{k}\not=\textbf{0}$ for some $k>N$, then \[\length(\G_{k})\not=g\length({\mathcal{H}}^{(N)}|_n),\] 
for all $n$ such that $n\le N$.
\item[III)] Row permutations on ${\mathcal{H}}^{(k-1)}|_n$, $k>n$, occur if and only if $\length(\G_k)<g\length({\mathcal{H}}^{(k-1)}|_{n})$ and $\G_k\not=\textbf{0}$, or equivalently, if and only if $g\length({\mathcal{H}}^{(k-1)}|_n)>g\length({\mathcal{H}}^{(k)}|_n)$. These are also necessary conditions for the implementation of the Jordan elimination; for if $\G_k$ eliminates an element of the row 
${\mathcal{H}}^{(k-1)}_n$ then, as $\length(\G_{k})<\length({\mathcal{H}}^{(k-1)}_n)$, row permutations necessarily follow.
Therefore, if $k>n$, the following statements are equivalent:
\begin{description}
\item[\textit{i})]   
${\mathcal{H}}^{(k-1)}|_n \not={\mathcal{H}}^{(k)}|_n$ (changes on ${\mathcal{H}}^{(k-1)}|_{n}$ occur).
\item[\textit{ii})] $\G_{k}\not=\textbf{0}$ and $\length(\G_{k})<g\length({\mathcal{H}}^{(k-1)}|_n)$.
\item[\textit{iii})] $g\length({\mathcal{H}}^{(k-1)}|_n)>g\length({\mathcal{H}}^{(k)}|_n)$.
    \end{description}
   As a generalization:
\end{description}
\begin{description}
\item[IV)] Changes on the matrix ${\mathcal{H}}^{(N)}|_n$, $n\le N$, will take place (by Jordan elimination and/or row permutations) if and only if there exists some $k$ with $k>N$ such that one of the following equivalent statements is confirmed:
\begin{description}
\item[\textit{i})] 
${\mathcal{H}}^{(N)}|_n\not={\mathcal{H}}^{(k)}|_n$.
\item[\textit{ii})]  $\G_{k}\not=\textbf{0}$ and   $\length(\G_{k})<g\length({\mathcal{H}}^{(N)}|_n)$.
\item[\textit{iii})]  $g\length({\mathcal{H}}^{(N)}|_n)>g\length({\mathcal{H}}^{(k)}|_n)$.
\end{description}
\end{description}
\begin{description}
\item[V)] Changes on the matrix ${\mathcal{H}}^{(N)}|_n$, $n\le N$, will not take place if and only if for every $k$ with $k>N$  one of the following equivalent statements is confirmed:
\begin{description}
\item[\textit{i})]  ${\mathcal{H}}^{(N)}|_n={\mathcal{H}}^{(k)}|_n$.
\item[\textit{ii})]  Either $\length(\G_{k})>g\length({\mathcal{H}}^{(N)}|_n)$  or $\G_{k}=\textbf{0}$.
\item[\textit{iii})]  $g\length({\mathcal{H}}^{(N)}|_n)=g\length({\mathcal{H}}^{(k)}|_n)$.
\end{description}
\end{description}}
\end{remark}
\subsection{The Main Theorem and the Chain of Matrices in QHF}
\label{sec:TheMainTheorem}
The fundamental theorem of this paper is presented below. It establishes that given an arbitrary number, say $n$, of rows constructed by the algorithm, there is a large enough number $N\ge n$ such that for all $k\ge N$ no further changes on these $n$ rows can take place by the process. 
\begin{theorem} \label{main}
For every $n\in \omega$ there exists $N=\delta_n\ge n$ such that 
\begin{equation}\label{main eq} {\mathcal{H}}^{(N)}\!\!\mid_n={\mathcal{H}}^{(k)}\!\!\mid_n\end{equation} 
for all $k\ge N$. 
\end{theorem}
\begin{proof} For every $n\in\omega$ define the set ${\mathcal{M}}_n\subset \omega$ such that $m\in {\mathcal{M}}_n$ if $m>n$, $\G_{m}\not=\textbf{0}$ and $\length(\G_{m})<g\length({\mathcal{H}}^{(m-1)}|_n)$. For any $n\in\omega$ there are two possible cases: 
\begin{description}
\item[\textit{Case i})]  ${\mathcal{M}}_n=\emptyset$. 
This hypothesis along with remark
\ref{remark1} (\textbf{II}) entails: for every $k>n$ either $\length(\G_{k})>g\length({\mathcal{H}}^{(k-1)}|_n)$ or $\G_{k}=\textbf{0}$. Choosing $N=n$, Eq. (\ref{main eq}) follows from remark \ref{remark1} (\textbf{V}) [$(ii) \Rightarrow (i)$]. 
\item [\textit{Case ii})]  ${\mathcal{M}}_n\not=\emptyset$. Call ${\mathcal{M}}_n=\{m_1,m_2,...\}$ with $m_1<m_2<...$. 
Notice that ${\mathcal{M}}_n$ could be either finite or infinite. Also call $n=m_0$. Therefore $m_0<m_1<...$.
As $m_i-1\ge m_{i-1}$, it follows (from inequality (\ref{inequality1})) that  $g\length({\mathcal{H}}^{(m_i-1)}|_n)\le g\length({\mathcal{H}}^{(m_{i-1})}|_n)$. As $m_i\in {\mathcal{M}}_n$ for $i\ge 1$, we infer that $\length(\G_{m_i})<g\length({\mathcal{H}}^{(m_i-1)}|_n)\le g\length({\mathcal{H}}^{(m_{i-1})}|_n)$. Thus for every  $m_i\in {\mathcal{M}}_n$ we have:
\[\G_{m_i}\not=\textbf{0} \ {\rm and}\ \length(\G_{m_i})<\length({\mathcal{H}}^{(m_{i-1})}|_n)\]
Applying remark \ref{remark1} (\textbf{IV}) with $N=m_{i-1}$ and $k=m_i$, it follows [from $(ii)\Rightarrow (iii)$] that $g\length({\mathcal{H}}^{(m_{i-1})}|_n)>g\length({\mathcal{H}}^{(m_i)}|_n)$.
Therefore,
\[g\length({\mathcal{H}}^{(m_0)}|_n)> g\length({\mathcal{H}}^{(m_1)}|_n)>g\length({\mathcal{H}}^{(m_2)}|_n)>...\ge 0.\]
Thus, the set ${\mathcal{M}}_n$ must be finite. 
Call $\delta_n=\max ({\mathcal{M}}_n)$.
Since $g\length({\mathcal{H}}^{(k)}|_n)=g\length({\mathcal{H}}^{(\delta_n)}|_n)$ for all $k\ge \delta_n$, Eq. (\ref{main eq}) follows from remark \ref{remark1} (\textbf{V}) [$(iii) \Rightarrow (i)$].
\end{description}
The proof of the theorem is complete.
\end{proof}
Throughout the rest of this paper we follow the notation of theorem \ref{main}. Moreover we shall consider $\delta_n\in\omega$ as the smallest integer greater than $n$ such that Eq. (\ref{main eq}) holds.
\begin{corollary} \label{chain}
The infinite Gauss-Jordan elimination algorithm generates a chain of submatrices in \emph{QHF}:
\begin{equation} \label{chain of submatrices of QHF} 
{\mathcal{H}}^{(\delta_0)}\!\!\mid_0 \sqsubset {\mathcal{H}}^{(\delta_1)}\!\!\mid_1\sqsubset ... \sqsubset {\mathcal{H}}^{(\delta_n)}\!\!\mid_n\sqsubset ...
\end{equation}
\end{corollary}
\begin{proof} As ${\mathcal{H}}^{(\delta_{n+1})}\!\!\mid_{n+1}= {\mathcal{H}}^{(k)}\!\!\mid_{n+1}$ for all $k\ge\delta_{n+1}$, it follows that ${\mathcal{H}}^{(\delta_{n+1})}\!\!\mid_n= {\mathcal{H}}^{(k)}\!\!\mid_n$ for all $k\ge\delta_{n+1}$. Thus $\delta_{n+1}$ satisfies Eq. (\ref{main eq}). As $\delta_n$ also satisfies Eq. (\ref{main eq}) it follows that $\delta_n\le\delta_{n+1}$. Now we can apply theorem \ref{main} with $k=\delta_{n+1}$ and $N=\delta_{n}$. We infer
\[{\mathcal{H}}^{(\delta_{n})}\!\!\mid_{n}={\mathcal{H}}^{(\delta_{n+1})}\!\!\mid_{n}\sqsubset {\mathcal{H}}^{(\delta_{n+1})}\!\!\mid_{n+1},\]
for all $n\in\omega$ and the chain of submatrices in (\ref{chain of submatrices of QHF}) follows.
\end{proof} 
By virtue of (\ref{chain of submatrices of QHF}), we define the matrix $\QHF=(\QHF_i)_{i\in\omega}$ as follows. For every $n\in\omega$, the row $\QHF_n$ of $\QHF$ is defined to be the $n$-row (the last row) of the matrix ${\mathcal{H}}^{(\delta_n)}\!\!\mid_n$. Formally $\QHF\in\RFM$. The matrix $(\QHF_i)_{0\leq i\leq n}$ (as sequence of rows) is
\begin{equation} \label{main2}
           (\QHF_i)_{0\leq i\leq n}={\mathcal{H}}^{(\delta_n)}\!\!\mid_n.
\end{equation}
Moreover $\QHF_n$, being a row of ${\mathcal{H}}^{(\delta_n)}\!\!\mid_n$, it is a finite linear combination of $\A_k$, thus
\begin{equation}\label{rows of H} \QHF_n=\sum_{k=0}^{M_n}q_{nk}\A_k\end{equation}
such that $q_{nM_n}\not=0$ and $0\le M_n\le \delta_n$. 
\begin{example} \label{sec:Example1} {\rm
The infinite Gaussian elimination algorithm under rightmost pivoting is illustrated with the use of symbolic computation through its application on the $\omega\times\omega$ row and column finite matrix:}\end{example}
\begin{equation} \label{example1}\!\!\!\!\!  
\A=\left(\begin{array}{ccccccc}
a_0 & b_0 & 1 &  0 &  0  &  0  & ...  \\
0 & a_1 & b_1 &  1 &  0  &  0  & ...  \\
0 & 0 & a_2   & b_2  &  1  &  0  & ... \\
\vdots & \vdots & \vdots     & \vdots    &  \vdots  &  \vdots  & \vdots\vdots\vdots 
\end{array} \right).
\end{equation}  
The matrix $\A$ in (\ref{example1}) is associated with the second order linear difference equation with variable coefficients (see section \ref{sec:ExplicitFundamentalSolutions}).
Call  ${\mathcal{H}}^{(0)}=(a_0,b_0,1,0,...)$.
Using as pivot the rightmost $1$ of ${\mathcal{H}}^{(0)}$, the algorithm eliminates the corresponding element of the second row of $\A$ yielding the matrix:
\[ {\mathcal{H}}^{(1)}=\left(
\begin{array}{ccccc}
 a_0 & b_0 & 1 & 0&...  \\
 -a_0 b_1 & a_1-b_0 b_1 & 0 & 1&... 
\end{array}
\right).\]
The row $(0,0,a_2,b_2,1,0...)$ is inserted below ${\mathcal{H}}^{(1)}$. The rightmost 1s of the rows of ${\mathcal{H}}^{(1)}$ are used as pivots to clear corresponding elements of the inserted row yielding the matrix: 
\begin{equation}\label{H2}
{\mathcal{H}}^{(2)}=\left(
\begin{array}{cccccc}
 a_0& b_0 & 1 & 0 & 0&...  \\
 -a_0 b_1 & a_1-b_0 b_1 & 0 & 1 & 0&... \\
 a_0 b_1b_2-a_0 a_1 & -a_2 b_0+b_0b_1b_2-a_1b_2& 0 & 0 & 1 & ... 
\end{array}
\right).
\end{equation}

The process continues ad infinitum, yielding the chain of submatrices: 
\begin{equation}\label{Chain Order-2} {\mathcal{H}}^{(0)} \sqsubset {\mathcal{H}}^{(1)} \sqsubset {\mathcal{H}}^{(2)} \sqsubset ...\sqsubset \QHF\end{equation}
In this case, $\delta_n=n$ for all $n\in \omega$. More generally speaking, in the case of $\omega\times\omega$ matrices in lower echelon form, as in (\ref{example1}), the Jordan elimination part of the algorithm as well as row permutations do not take place. Upon algorithm completion the $\omega\times\omega$ matrix $\QHF$ is in  HF (or LRREF). As $\A$ and $\QHF$ are of zero left- nullities and the rows of $\QHF$ are finite linear combinations of the rows of $\A$ and vice-versa $\A, \QHF$ are left associates (for a general proof see the next paragraph).

In contrast to the LRREF of $\A$, there is no URREF for simple row and column finite matrices, such as $\A$, preserving left association. To see this consider the matrix 
\[\B=\left(\begin{array}{cccccc}
1 & 1 & 0 &  0   &  0      & ...  \\
0 &  1  & 1 & 0  &  0      & ...  \\
0 &  0  & 1   & 1  &  0   & ...  \vspace{-0.05in}\\
. & .   & .   & .    &  .   & ...  \vspace{-0.1in}\\
. & .   & .   & .    &  .   & ...  \vspace{-0.1in}\\
. & . & .   & .    &  .  &  ... 
\end{array} \right).\]
Formally $\B$ is the coefficient matrix of the first order linear difference equation: $y_{n+1}+y_n=0$, $n\in\naturals$.  As the URREF of $\B$ must have 0s above and below the leading 1s, the URREF of $\B$ must by the identity $\omega\times\omega$ matrix $\I$. But $\B$ and $\I$ are not left associates. 
For if otherwise, $\RS(\B)$ would equal $\RS(\I)$. Consequently, there would be a finite linear combination of the rows of $\B$ generating the row $(1,0,0,...)$ of $\I$. But this is impossible because all the finite linear combinations of rows of $\B$ yield row-lengths greater than or equal to $1$, while the length of $(1,0,0,...)$ is $0$.

The non-existence of row-equivalent URREFs for an extensive class of row-finite matrices prevents the implementation of the infinite Gauss-Jordan elimination in solving row-finite systems. Yet, even in this case the algorithm works by applying the Jordan elimination to $\B$. It provides the sequence of reduced matrices:
\[\begin{array}{ll}
\C^{(0)}=&\left(\begin{array}{ccrccc}
1 & 1& 0& 0 & 0 &...\end{array}\right),\vspace{0.1in}\\
\C^{(1)}=&\left(\begin{array}{ccrccc}
 1 & 0 & -1 & 0& 0& ... \\
 0 & 1 & 1 & 0 & 0& ...
\end{array}
\right)\end{array},\]
\[\begin{array}{ll}
\C^{(2)}=&\left(\begin{array}{lllrcc}
 1 & 0 & 0 & 1  &0&... \\
 0 & 1 & 0 & -1 & 0&... \\
 0 & 0 & 1 & 1 & 0&...
\end{array}\right),...
\end{array}\]
Unlike the sequence in (\ref{Chain Order-2}), the sequence $(\C^{(n)})_{n\in \omega}$ does not form a chain of submatrices, because as the process continues, the entries above the leading 1s become 0s and thus the first rows can never coincide. Moreover, all columns are progressively replaced by columns of the identity matrix. Upon algorithm completion, the identity $\omega\times\omega$ matrix is reached. 
The latter agrees with our previous conclusion that the URREF of $\B$ should be $\I$.
We thus need an infinite number of linear combinations to produce the row $(1,0,0,...)$.

As infinite sequences of row elementary operations may or may not preserve row equivalence, we conclude that in the case of row-finite matrices the primary definition of row equivalence through sequences of row elementary operations is not met. 

\subsection{Left-Association}
\label{sec:Left-Association}
In theorem \ref{A is row-equivalent to H} \hspace{0.001in} it will be shown that the infinite Gauss-Jodran elimination algorithm, applied to an arbitrary $\A\in\RFM$ under rightmost pivoting, constructs a matrix $\QHF\in\RFM$ in QHF, such that $\A$ and $\QHF$ are left associates.
\begin{definition}\label{full-length sequence} Let $Y$ be a nontrivial subspace of $\complex^{(\omega)}$ and $\emptyset\varsubsetneqq J\subseteq \omega$. A sequence $z=(z^{(n)})_{n\in J}$ in $Y\setminus \{0\}$ is said to be a \emph{full-length sequence} of $Y$, if for every $y\in Y\setminus \{0\}$ there exists some $k\in J$ such that $\length(z^{(k)})=\length(y)$.
\end{definition}
\begin{proposition} \label{full-length basis} Let $z=(z^{(n)})_{n\in J}$ be a full-length sequence of a non-trivial subspace $Y$ of $\complex^{(\omega)}$. If $\length\circ z: J\mapsto \omega$ is injective, then $z$ is a Hamel basis of $Y$.
\end{proposition}
\begin{proof} First we show that $z$ spans $Y$. Let $y=(y_{0},y_{1},...,y_{m_0},0,0,...)$ be an arbitrary element of $Y$ such that $y_{m_0}\not=0$, that is $\length(y)=m_0\ge 0$. We next show that $y\in \Span(z)$. By definition \ref{full-length sequence}, there is a $n_0\in J$ such that $\length(z^{(n_0)})=\length(y)=m_0$. Thus $z^{(n_0)}=(z^{(n_0)}_0,z^{(n_0)}_1,...,z^{(n_0)}_{m_0},0,0,...)$\vspace{0.05in} with $z^{(n_0)}_{m_0}\not=0$. Call $\displaystyle a_{m_0}=\frac{y_{m_0}}{z^{(n_0)}_{m_0}}$ and $m_1=\length(y-a_{m_0}z^{(n_0)})$. Obviously $-1\le m_1<m_0$. If $m_1=-1$, that is $y-a_{m_0}z^{(n_0)}=0$, then $y\in \Span(z)$. Otherwise, as $y-a_{m_0}z^{(n_0)}\in Y\setminus \{0\}$ definition \ref{full-length sequence} implies that there is some $n_1\in J$ such that $m_1=\length(z^{(n_1)})\ge 0$. Proceeding in this way, we construct a sequence of integers such that $m_0>m_1>m_2>...\ge -1$. Thus, there exists some $k\in\omega$ such that $m_k=-1$, that is
\[y-a_{m_0}z^{(n_0)}-a_{m_1}z^{(n_1)}-...-a_{m_{k-1}}z^{(n_{k-1})}=0,\]
as asserted. 

It remains to be shown that $z$ is linearly independent. Let $(\e_n)_{n\in\omega}$ be the canonical basis of $\complex^{(\omega)}$. Let also $Q=\{z^{(i_0)},...,z^{(i_m)}\}$ be a finite subset of $z$ and 
${\mathcal{L}}_Q=\{\length(z^{(i_0)}),...,\length(z^{(i_m)})\}$. Without loss of generality, since $\length \circ z$ is injective, we assume that $\length(z^{(i_0)})<...<\length(z^{(i_m)})$ and we define the sequence:
\[\beta_{n}=\left\{\begin{array}{lll} \e_n  & {\rm if} \ n\not\in{\mathcal{L}}_Q\\
 &        & 0\le n\le \length(z^{(i_m)}) \\
   z^{(i_k)} & {\rm if} \ n=\length(z^{(i_k)}) \end{array}\right.  \]
Consider the $(\length(z^{(i_m)})+1)\times\omega$ matrix $\B$ with rows $\beta_n\in\complex^{(\omega)}$, $0\le n\le \length(z^{(i_m)})$. Call $\B_m$ the submatrix of $\B$ consisting of the first $\length(z^{(i_m)})+1$ columns of $\B$. Evidently $\B_{m}$ is a nonsingular $(\length(z^{(i_m)})+1)\times(\length(z^{(i_m)})+1)$ matrix, as being lower triangular having non-zero diagonal. The Casoratian $W(0)$ of $\B$ is the determinant of $\B_{m}$, whence $W(0)\not=0$. It follows that the rows of $\B$ are linearly independent. Thus $Q$, being a subset of a linearly independent set, is linearly independent. Since every finite subset of $z$ is linearly independent so is $z$. Accordingly $z$ is a Hamel basis of $Y$, as required.
\end{proof}
The Hamel basis $z$ in proposition \ref{full-length basis} will be called \emph{full-length basis} of $Y$. The set of non-zero rows of a row-finite matrix, say $\B$, in lower echelon form meets the conditions of proposition \ref{full-length basis} yielding a Hamel basis of $\RS(\B)$. Taking into account that a QHF is in lower echelon form,  the subsequent result is equivalent to the quasi-uniqueness of $\QHF$. 
\begin{corollary} \label{Hermite basis}
The non-zero rows of a \emph{QHF}, $\QHF$, of $\A$ is a Hamel basis of $\RS(\A)$ uniquely associated with $\A$, called Hermite basis of $\RS(\A)$.
\end{corollary}
From corollary \ref{Hermite basis} the next result immediately follows. \begin{corollary} \label{cor. row space} The row spaces of $\A$ and $\QHF$ coincide: $\RS(\A)=\RS(\QHF)$. 
\end{corollary}
The following proposition concerns the Gaussian elimination algorithm, applied solely to a row-finite matrix, as illustrated in example \ref{sec:Example1}.
\begin{proposition} \label{A is row-equivalent to G}
The infinite Gaussian elimination applied  to $\A\in\RFM$ results in a matrix $\G\in\RFM$ possessing the following properties:
\begin{description} 
\item[\textit{i})]  $\G$ and $\A$ are left associates.
\item[\textit{ii})]  The sequence of lengths of the non-zero rows of $\G$ is injective. 
\item[\textit{iii})]  The non-zero rows of $\G$ form a full-length sequence of $\RS(\A)$.
\item[\textit{iv})]  The non-zero rows of $\G$ form a full-length basis of $\RS(\A)$.
\end{description}
\end{proposition}
\begin{proof}
\begin{description} 
\item[\textit{i})]
The rows $\{\G_n\}_{n\in\omega}$ of the resulting matrix $\G=(g_{nm})_{(n,m)\in\omega\times\omega}$\vspace{0.05in} are linear combinations of preceding rows of $\A$, that is $\G_n=\sum_{k=0}^nc_{nk}\A_k$ with $c_{nn}\not=0$.
Define the row-finite matrix $\C=(c_{nm})_{(n,m)\in\omega\times\omega}$ with rows  $\C_n=(c_{n0},c_{n1},...,c_{nn},0,0,...)=\sum_{k=0}^nc_{nk}\I_k$ for all $n\in \omega$. That is, the rows
$\C_n$ are produced by applying the same sequence of row elementary operations to $\I\in\RFM$, as those transforming $\A$ to $\G$. As $c_{nn}\not=0$ the matrix $\C$ is a lower triangular whose main diagonal consists of non-zero entries. Thus, $\C$ is nonsingular. Moreover, as $g_{nm}=\sum_{k=0}^nc_{nk}a_{km}$ it follows that $\C\cdot \A=\G$. The latter entails that $\G, \A$ are left associates.
\item[\textit{ii})] Let $(\G_j)_{j\in J}$, be the sequence of non-zero rows of $\G$ and $\length_j=\length(\G_j)$. Consider $m, n\in J$ such that $m<n $. The entries $g_{nk}$ of $\G_n$ are positioned below the entries of $\G_m$. The entry $g_{n\length_m}=0$, due to the elimination by the pivot element $g_{m\length_m}$. On the contrary, assume that $\length_n=\length_m$.
Since $g_{n\length_n}$ is the rightmost non-zero element of $\G_n$ we have $g_{n\length_n}\not=0$. The assumption $\length_n=\length_m$ leads to the contradictory statement: $0\not=g_{n\length_n}=g_{n\length_m}=0$, whence the assertion.
\item[\textit{iii})] Let $\R\in\RS(\A)$ and $\R\not=\textbf{0}$. As $\G$ and $\A$ are left associates, it follows from theorem \ref{theorem of row equivalence} that $\RS(\A)=\RS(\G)$. Hence, the non-zero rows of $\G$ span $\RS(\A)$ and therefore $\R=\sum_{k=0}^n \alpha_k\G_{k}$. The foregoing statement (ii) entails that among the terms $\G_0,\G_1,...,\G_n$ in the sum of $\R$, there is a unique $\G_{m}$ with $0\le m\le n$ such that $\length(\G_{m})>\length(\G_{k})$ for all $0\le k\le n$ with $k\not= m$. Therefore $\displaystyle\length(\G_{m})>\length(\sum_{\substack{k=0 \\ k\not=m}}^n \alpha_k\G_{k})$, whence $\displaystyle\length(\G_{m})=\length(\sum_{k=0}^n \alpha_k\G_{k})=\length(\R)$, as required.
\item[\textit{iv})] As the conditions of proposition \ref{full-length basis} are satisfied by statements (\textit{ii}) and (\textit{iii}), the result follows.
\end{description}
\end{proof}
As a consequence, the infinite Gaussian elimination is a rule, alternative to the axiom of countable choice, which
constructs a full-length basis of $\RS(\A)$.

\begin{proposition} \label{Bourbaki21} 
\begin{description} The following statements hold:
\item[\textit{i})] $\Q\in \RFM$ is nonsingular if and only if the rows of $\Q$ form a Hamel basis of $\complex^{(\omega)}$ \emph{(see Fulkerson~\cite[Corollary 2.4 pp. 15]{Fu:Th})}.
 \item[\textit{ii})] 
 Let $\A\in \RFM$. Let also $\{\B_{j}\}_{j\in J}$ be a Hamel basis of $\RS(\A)$ and $\alpha:\complex^{(\omega)}\ni x\mapsto \alpha(x)=x\cdot \A \in\complex^{(\omega)}$ be the endomorphism induced by $\A$. If $(\Q_j)_{j\in J}$, $J\subset \omega$, is a sequence in $\complex^{(\omega)}$ such that $\alpha(\Q_j)=\B_j$ \emph{(}or $\Q_j\cdot\A=\B_j$\emph{)} for all $j\in J$, then $(\Q_j)_{j\in J}$ is a Hamel basis of a complementary space of $\LNS(\A)$ \emph{(see Bourbaki~\cite[Proposition 21 pp. 218]{Bou:Alg})}. 
\end{description}
\end{proposition}
\begin{theorem} \label{A is row-equivalent to H}
The composite of the infinite sequence of row elementary operations, determined by the infinite Gauss-Jordan elimination algorithm for reducing $\A$ to $\QHF$, is represented by a nonsingular row-finite matrix $\Q$ such that $\Q\cdot\A=\QHF$, namely $\A$, $\QHF$ are left associates. Moreover, the matrix $\QHF$ is a \emph{QHF} of $\A$.
\end{theorem}  
\begin{proof} Let $\A=(a_{nm})_{(n,m)\in\omega\times\omega}$, $\QHF=(h_{nm})_{(n,m)\in\omega\times\omega}$, $\Q=(q_{nm})_{(n,m)\in\omega\times\omega}$. In view of Eq. (\ref{rows of H}), we can define \[\Q_n=\sum_{k=0}^{M_n}q_{n,k}\I_k=(q_{n,0},q_{n,1},...,q_{n,M_n},0,0,...),\  \ q_{n,M_n}\not=0. \]
The sequence $\Q=(\Q_n)_{n\in\omega}$\vspace{0.03in} can be viewed as a row-finite matrix whose rows are $\Q_n$. It follows from Eq. (\ref{rows of H}) that $h_{n,m}=\sum_{k=0}^{M_n}q_{n,k}a_{k,m}$, whence $\Q\cdot \A=\QHF$. 
Thus, $\Q$ represents the composite of the infinite sequence of row elementary operations used by the infinite Gauss-Jordan elimination algorithm for reducing $\A$ to $\QHF$. 

In order to show that $\Q$ is nonsingular, it suffices to show that $\{\Q_n, n\in\omega\}$ is a Hamel basis of $\complex^{(\omega)}$ (see proposition \ref{Bourbaki21} (i)). As zero rows are exclusively produced by the Gaussian elimination and their positions remain unchanged under row permutations, it follows that each zero row $\QHF_w$ is a linear combination of preceding rows of $\A_w$ including row $\A_w$, that is $\QHF_w=\sum_{k=0}^{w}q_{w,k}\A_{k}$. Moreover,
$\Q_w=\sum_{k=0}^{w}q_{w,k}\I_k=(q_{w,0},q_{w,1},...,q_{w,w},0,0,...)$ with $q_{w,w}\not=0$ and $\Q_w\in\LNS(\A)$.
Following the notation of the proof of proposition \ref{A is row-equivalent to G},  $\G_w=\QHF_w$, whence $\C_w=\Q_w$ for all $w\in W$ and $\C_w\in\LNS(\A)$. As $\C$ is nonsingular, proposition \ref{Bourbaki21} (i) implies that $\{\C_n\}_{n\in\omega}$ is a Hamel basis of $\complex^{(\omega)}$. Let $J=\omega\setminus W$. The subset $\{\C_j\}_{j\in J}$ of $\{\C_n\}_{n\in\omega}$ satisfies $\C_j\cdot \A=\G_j$ for all $j\in J$. As  $\{\G_j\}_{j\in J}$ is a Hamel basis of $\RS(\A)$ (see proposition \ref{A is row-equivalent to G}), it follows from proposition \ref{Bourbaki21} (ii) that $\{\C_j\}_{j\in J}$ is a Hamel basis of a complementary space of $\LNS(\A)$. Thus, $\{\Q_w\}_{w\in W}$ (or $\{\C_w\}_{w\in W}$) is a Hamel basis of $\LNS(\A)$. As $\Q_j\cdot \A=\QHF_j$ and $\{\QHF_j\}_{j\in J}$ is a Hamel basis of $\RS(\A)$ (see corollary \ref{Hermite basis}), proposition \ref{Bourbaki21} (ii) entails that $\{\Q_j, j\in J\}$ is a Hamel basis of a complementary space of $\LNS(\A)$, whence $\{\Q_n, n\in\omega\}=\{\Q_w, w\in W\}\cup\{\Q_j, j\in J\}$ is a Hamel basis of $\complex^{(\omega)}$, as claimed. 

In order to show that $\QHF$ is a QHF  of $\A$,  consider any two non-zero rows $\QHF_i$ and $\QHF_n$ of $\QHF$ such that $i<n$. According to Eq. (\ref{main2}), there exists $N=\delta_n$ with $N\ge n$ such that $\QHF_i$ and $\QHF_n$ are rows of ${\mathcal{H}}^{(N)}\!\!\mid_{n}$. Since ${\mathcal{H}}^{(N)}\!\!\mid_{n}$ is in QHF the postulates of definition \ref{QHF} are fulfilled. 
\end{proof}
\begin{remark} \label{alternative proof}
{\rm Theorem \ref{A is row-equivalent to H} shows, among others, that $\QHF$ and $\A$ are left associates. An alternative and considerably shorter proof of $\QHF$ and $\A$ left-association is shown hereby. Remark \ref{remark1}(I) implies that $\nul(\QHF)=\nul(\G)$. Proposition \ref{A is row-equivalent to G} entails that $\G$ and $\A$ are left associates, whence $\nul(\A)=\nul(\G)=\nul(\QHF)$. Since $\RS(\A)=\RS(\QHF)$ (see corollary \ref{cor. row space}), theorem \ref{theorem of row equivalence} implies that $\QHF$ and $\A$ are left associates, as required.}
\end{remark}
As in the finite dimensional case, the existence of a QHF namely $\QHF$ of $\A\in \RFM$, is deduced from the constructiveness of $\QHF$. Even though $\QHF$ is obtained by performing an infinite sequence of elementary row operations on $\A$, each row of $\QHF$ is a finite number of linear combinations of the rows of $\A$.
\section{Construction of the General Solution Sequence for Row-Finite Linear Systems}
\label{sec:ConstructionOfRow-FiniteSystemSolutions} 
Let $\QHF$ be the QHF of the coefficient matrix $\A$ of Eq. (\ref{MQLR}) constructed by the infinite Gauss-Jordan elimination algorithm. Let also $\Q$ be the nonsingular matrix representing the composite of elementary row operations transforming $\A$ to $\QHF$, i.e. $\Q\cdot \A=\QHF$. Multiplying both sides of Eq. (\ref{MQLR}) by $\Q$ we have: $\A \cdot y= g \Longrightarrow  \QHF\cdot y= \Q\cdot g$.
As $\Q$ is nonsingular the converse statement also holds by multiplying both sides of $\QHF\cdot y= \Q\cdot g$ by $\Q^{-1}$ thus:\vspace{-0.05in}
\begin{equation} \label{equivalent-systems}
\A \cdot y= g \Longleftrightarrow  \QHF\cdot y= \Q\cdot g.
\end{equation}
Let us call $\textbf{k} \stackrel{{\rm def}}{=}\Q\cdot g\in \Com$. According to Eq. (\ref{equivalent-systems}), the system  (\ref{MQLR}) is equivalent to the system
\begin{equation} \label{equivalent-system}
\QHF\cdot y= \textbf{k},
\end{equation}
meaning that any solution of Eq. (\ref{MQLR}) is a solution of Eq. (\ref{equivalent-system}) and vise versa.
\subsection{General Homogeneous Solution Sequence}
\label{sec:HomogeneousSolution} 
Throughout the paper, the cardinality of a set $X$ is denoted by  $\card(X)$. The set of all lengths of the elements in $\RS(\A)\setminus \{0\}$ will be denoted by $\lengths_{\A}$.
The set $\{\mu_0,\mu_1,\mu_2,...\}$ of lengths of the non-zero rows of a QHF $\QHF$ of $\A$ and the set $\lengths_{\A}$ coincide. This is due to
the fact that the sequence of non-zero rows of $\QHF$ is a full-length basis of $\RS(\A)$ (see the discussion following proposition \ref{full-length basis}).
The set $\clengths_{\A}=\omega\setminus \lengths_{\A}$ is the set of \emph{inaccessible row-lengths} of $\RS(\A)$. The co-dimension of $\RS(\A)$, as a subspace of $\complex^{(\omega)}$, is called \emph{deficiency} (or \emph{defect}) of $\A$ and is denoted by $\defic(\A)$. Since $\RS(\A)=\RS(\QHF)$, it follows that $\defic(\A)=\defic(\QHF)$.
\begin{proposition}\label{deficiency} The deficiency of a row-finite matrix $\A$ coincides with the number of inaccessible row lengths of the row space of $\A$:
\begin{equation}\label{deficiency of A} \defic(\A)=\card(\clengths_{\A})\end{equation}
\end{proposition}
\begin{proof} Let $\{\QHF_j\}_{j\in J}$ be the sequence of non-zero rows of $\QHF$ and $(\e_n)_{n\in\omega}$ be the canonical basis  of $\complex^{(\omega)}$.
The sequence$(\varphi_{n})_{n\in\omega}$ is defined as follows:
\[\varphi_{n}=\left\{\begin{array}{lll} \QHF_j      & {\rm if} \ \length(\QHF_j)=n   \\                                    &               & n\in\omega \\
\e_n  & {\rm otherwise} \end{array}\right.  \]
The sequence $(\varphi_{n})_{n\in\omega}$ is a full-length sequence of $\complex^{(\omega)}$ and $(\length(\varphi_{n}))_{n\in \omega}$ is injective. By proposition \ref{full-length basis}, the set $\Phi=\{\varphi_{n}: n\in\omega\}=\{\QHF_j\}_{j\in J}\cup \{\e_s\}_{s\in \clengths_{\A}}$ is a Hamel basis of $\complex^{(\omega)}$. As $\Span(\QHF_j)_{j\in J}=\RS(\QHF)=\RS(\A)$, we infer that 
\[\complex^{(\omega)}=\Span(\Phi)=\RS(\QHF)\oplus \Span(\e_s)_{s\in \clengths_{\A}}= \RS(\A)\oplus \Span(\e_s)_{s\in \clengths_{\A}}.\]
This completes the proof.
\end{proof}
Considering $\A$ as left operator, the space of homogeneous solutions of system  (\ref{MQLR}) (the right-null space of  $\A$) will be denoted as: 
\[\RNS(\A)=\{y\in \Com: \A\cdot y=\textbf{0}\}.\] 
By virtue of Eq. (\ref{equivalent-systems}), the system $\A\cdot y=\textbf{0}$ is equivalent to the system $\QHF\cdot y= \textbf{0}$, whence $\RNS(\A)=\RNS(\QHF)$. 
Let \vspace{0.02in}$\QHF_{j_i}=(h_{j_i0},h_{j_i1},...,h_{j_i\; \mu_i-1},1,0,0,...)$ be the $j_i$ non-zero row of $\QHF$ and $y$ be a homogeneous solution of Eq. (\ref{MQLR}), that is $y=(y_0,y_1,y_2,...,)^T\in \RNS(\A)$, where ``$T$" stands for transposition. 

In view of Eq. (\ref{QHF}) for every $i\in\omega$ we have:
\[ \begin{array} {ll} 
\QHF_{j_i} \cdot y &= (h_{j_i0},h_{j_i1},...,h_{j_i\mu_i-1},h_{j_i\mu_i},0,0,...)\cdot (y_0,y_1,...,y_{\mu_i},y_{\mu_i+1},...)^T \\
&= (h_{j_i 0},h_{j_i 1},...,h_{j_i\mu_i-1},h_{j_i\mu_i})\cdot (y_0,y_1,...,y_{\mu_i})^T 
 = \displaystyle \sum^{\mu_i}_{k=0} h_{j_ik} y_k. \vspace{-0.1in}
\end{array}
\]
Taking into account that $h_{j_i\mu_i}=1$ (rightmost $1$) we have:
\[\QHF_{j_i} \cdot y =0 \Longleftrightarrow  y_{\mu_i}+\sum^{\mu_i-1}_{k=0} h_{j_ik} y_k =0 \Longleftrightarrow 
y_{\mu_i}=-\sum^{\mu_i-1}_{k=0} h_{j_ik} y_k. 
\]
It follows that the general homogeneous solution of Eq. (\ref{MQLR}) is a sequence in $\Com$ of the form
\begin{equation} \label{homogeneous solution 2} \begin{array}{llll}y_{\mathcal{H}}=\!\!&(y_0, ...,y_{\mu_0-1}, -\!\!\displaystyle \sum^{\mu_0-1}_{k=0}\! h_{j_0k} y_k,&\hspace{-0.1in} y_{\mu_0+1}, ...,&\hspace{-0.15in} y_{\mu_1-1}, -\!\!\displaystyle\sum^{\mu_1-1}_{k=0} \! h_{j_ik} y_k,y_{\mu_1+1},...,\\
& &\hspace{-0.1in} y_{\mu_{i-1}+1},...,&\hspace{-0.1in}y_{\mu_i-1}, -\!\!\displaystyle\sum^{\mu_i-1}_{k=0} \! h_{j_ik} y_k,y_{\mu_i+1},\! ...)^T.
\end{array}\end{equation}
where $y_k$ in (\ref{homogeneous solution 2}) are free constants.
 

Definition\vspace{0.02in} \ref{definition-Hermite basis} (iii) entails that 
$h_{j_i\mu_{i-n}}=0$ for $1\le n\le i$. Therefore, every term in the sum $\sum^{\mu_i-1}_{k=0} h_{j_ik} y_k$ of Eq. (\ref{homogeneous solution 2}) indexed by $k\in \lengths_{\A}$ is zero. Hence, the indexing set of the free constants $y_k$ in Eq. (\ref{homogeneous solution 2}) is $\clengths_{\A}$.  In view of Eq. (\ref{deficiency of A}), the number of free constants $y_k$ in Eq. (\ref{homogeneous solution 2}) equals the deficiency of $\A$. Thus, nontrivial homogeneous solutions exist if and only if $\defic(\A)>0$, that is $\clengths_{\A}\not=\emptyset$ (or $\lengths_{\A}\subsetneqq\omega$). 
\subsection{General Solution Sequence}
\label{sec:GeneralSolution}
The components of $\textbf{k}$ (as defined at the beginning of this section) are of the form: $\textbf{k}_i=\Q_i\cdot g$. Let $y\in \Com$. Then $\QHF_w\cdot y=0$ for all $w\in W$ (recall that $W$ denotes the indexing set of zero rows of $\QHF$). Furthermore, if $y$ is a solution of Eq. (\ref{MQLR}), then
\begin{equation} \label{consistent}\textbf{k}_w=0, \ \  {\rm for \ all} \ \ w\in W.\end{equation}
Let us define the sequence (recall that $J=\{j_0<j_1<...\}$ and $J=\omega\setminus W$) 
\begin{equation} \label{particular-solution}
\begin{array} {llll}
y_{\mathcal {P}}=(0,0,...,0,\!\!\!\!\!& \textbf{k}_{j_0},0,0,...,0,\!\!\!\!\!& \textbf{k}_{j_1},0,...,0,\!\!\!\!\!&\textbf{k}_{j_i},0,0,...)^T, \\
& \uparrow                  & \uparrow                 & \uparrow        \vspace{-0.05in}\\
& \mu_0                    & \mu_1                     & \mu_i
               \end{array}  
\end{equation}
where the component $\textbf{k}_{j_i}$ of $\textbf{k}$ has the position $\mu_i\in\lengths_{\A}$ in $y_{\mathcal {P}}$. According to the definition of QHFs (Definition \ref{definition-Hermite basis} (\textit{iii})), every non-zero row $\QHF_{j_i}$ of $\QHF$, at the positions $\mu_k$ for $k\not=i$ 
has zero entries, and at the position $k=i$ has the entry $1$, we infer that $\QHF_{j_i}\cdot\ y_{\mathcal {P}}=\textbf{k}_{j_i}$ for all $i\in\omega$.  

As a consequence, $y_{\mathcal {P}}$ in (\ref{particular-solution}) is a particular solution sequence of Eq. (\ref{MQLR}). Thus, a necessary and sufficient condition for the Eq. (\ref{MQLR}) to be consistent (it has a solution) is that Eq. (\ref{consistent}) holds.
The general solution $y_{\mathcal {G}}=(y_m)^T_{m}$ of Eq. (\ref{MQLR}) is $y_{\mathcal {G}}=y_{\mathcal {P}}+y_{\mathcal{H}}$, that is
\begin{equation} \label{general solution 2}
y_{\mathcal {G}}\!=\!(y_0,\! ...,y_{\mu_0-1}, \textbf{k}_{j_0}\!\!-\!\!\!\sum^{\mu_0-1}_{k=0}\!\!  h_{j_0k} y_k, y_{\mu_0+1},\! ...,y_{\mu_i-1}, \textbf{k}_{j_i}\!\!-\!\!\!\sum^{\mu_i-1}_{k=0}\!\! h_{j_ik} y_k,y_{\mu_i+1},\! ...)^T, 
\end{equation}
where the $y_k$s are free constants, while $h_{j_ik}$s and $\textbf{k}_{j_i}$s are all constructed by the infinite Gauss-Jordan elimination algorithm.
\section{Linear Difference Equations with Variable Coefficients\\ of Regular Order}
\label{sec:DifferenceEquations}
The general form of a linear difference equation with variable coefficients is given by
\begin{equation} \label{QLR}
a_{n,N+n} y_{N+n}+a_{n,N+n-1} y_{N+n-1}+...+a_{n,1}y_{1}+a_{n,0}y_{0}=g_n,\  \ n\ge N,
\end{equation}
where $N$ is a non-negative fixed integer and $a_{n,i}, g_n$ are arbitrary complex valued functions.

The sequence of equations in (\ref{QLR}) can be written as an infinite linear system with the  following row-finite coefficient matrix:
\begin{equation} \label{CMQLR}
\A\!\!=\!\!\left(\!\!\!\! \begin{array}{cccccccccccc}
 a_{0,0} \!\! &\!\!\!  a_{0,1}  \!\!   &\!\!\! ... \!\!\! &\!\!  a_{0, N-1}\!\!    &\!\!\!   a_{0, N}     \! &\!\!\!  0  \!\!  &\!\!\!  ... \!\!     & \!\!\!   0   \!\!\!\!         &\!\!  0\!\!\!\!\! &...\\
 a_{1,0} \!\! &\!\!\!  a_{1,1}  \!\!   &\!\!\! ... \!\!\!    &\!\!  a_{1, N-1}\!\!   &\!\!\!  a_{1,N} \! &\!\!\! a_{1, N+1} \!\!      &\!\!\! ... \!\!    & \!\!\!     0 \!\!\!\!          &\!\!   0 \!\!\!\!\! &... \\
 .       \!\! &\!\!\!   .            \!\!   & \!\!\! ... \!\!\!   &\!\!        .    \!\!  &\!\!\!    .      \! &\!\!\!  . \!\!   & \!\!\! ... \!\!       &\!\!\!     .       \!\!\!\!    &\!\!  . \!\!\!\!\! & ...  \vspace{-0.1in}\\
  .       \!\! &\!\!\!   .             \!\!  & \!\!\! ... \!\!\!   &\!\!       .    \!\!  &\!\!\!    .      \! &\!\!\!  . \!\! & \!\!\! ... \!\!          &\!\!\!     .    \!\!\!\!     &\!\!  . \!\!\!\!\!&...\vspace{-0.1in}\\
  .       \!\! &\!\!\!   .            \!\!  &\!\!\! ... \!\!\!   &\!\!        .    \!\!  &\!\!\!    .      \! &\!\!\!  . \!\! & \!\!\! ... \!\!         &\!\!\!     . \!\!\!\!         &\!\!  . \!\!\!\!\!&...  \\
a_{n-1, 0}\!\! &\!\!\! a_{n-1, 1} \!\!  &\!\!\! ... \!\!\!   &\!\!  a_{n-1, N-1}\!\! &\!\!\! a_{n-1, N} \! &\!\!\! a_{n-1, N+1}\!\! &\!\!\! ... \!\!  &\!\!\!   a_{n-1, N+n-1}         \!\!\!\!    &\!\!   0\!\!\!\!\!& ... \\
 a_{n0}  \!\!  &\!\!\! a_{n1}  \!\! &\!\!\!  ... \!\!\! &\!\! a_{n, N-1} \!\!     &\!\!\! a_{n, N} \!\! &\!\!\!   a_{n, N+1}  \!\! &\!\!\! ... \!\!   &\!\!\! a_{n,N+n-1}   \!\!\!\!    &\!\!   a_{n,N+n}\!\!\!\!\!& ...\\
 .       \!\! &\!\!\!   .            \!\!   & \!\!\! ... \!\!\!   &\!\!        .    \!\!  &\!\!\!    .      \! &\!\!\!  . \!\!   & \!\!\! ... \!\!       &\!\!\!     .       \!\!\!\!     &\!\!   .\!\!\!\!\! & ...  \vspace{-0.1in}\\
  .       \!\! &\!\!\!   .             \!\!  & \!\!\! ... \!\!\!   &\!\!     .    \!\!  &\!\!\!    .      \! &\!\!\!  . \!\! & \!\!\! ... \!\!          &\!\!\!     .    \!\!\!\!     &\!\!   .\!\!\!\!\!& ...  \vspace{-0.1in}\\
  .       \!\! &\!\!\!   .            \!\!  &\!\!\! ... \!\!\!   &\!\!        .    \!\!  &\!\!\!    .      \! &\!\!\!  . \!\! & \!\!\! ... \!\!         &\!\!\!     . \!\!\!\!    &\!\!   .\!\!\!\!\!& ...   
\end{array}\!\!\!  \right). \end{equation}

When the row-lengths of $\A$  vary irregularly, ($a_{n,N+n}=0$\vspace{0.04in} for some $n\in \naturals$), the equation (\ref{QLR}) is a \emph{linear difference equation of irregular order}. This type of equation is treated as a row-finite system by implementing the infinite Gauss-Jordan elimination algorithm in its full extent.

On the other hand ($a_{n,N+n}\not=0$ for all $n\in \naturals$), the equation (\ref{QLR}) is a linear difference equation of regular order (RO-LDEVC), which contains the following types of recurrences:
\begin{itemize}
\item If $a_{n,N+n}\not=0$ for all $n\in \naturals$, $a_{n,n}\not=0$ for some $n\in \naturals$ and $a_{n,i}=0$ for all $i,n\in \naturals$ such that $0\le i<n$, the equation (\ref{QLR}) is a \emph{linear difference equation with variable coefficients of order $N$} (see subsection \ref{sec:NOrderLinearDifferenceEquationWithVariableCoefficients}). 
\item If $N=0$ and $a_{n,n}\not=0$ for all $n\in \naturals$, the equation (\ref{QLR}) represents the commonly occurred linear difference equation, named by Mallik (~\cite{Ma:Ex}, 1998),  \emph{linear difference equation of unbounded order}. 
\item If $a_{n,N+n}\not=0$ for all $n\in \naturals$, $a_{n,0}\not=0$ for some $n\in \naturals$, the equation (\ref{QLR}) will be referred to as \emph{ascending order linear difference equation of index} $N$. 
\end{itemize}

Equations of regular order yield  row-finite coefficient matrices of the following types:
\begin{itemize}
\item Equations of constant order ($N$) are associated with infinite band matrices (row and column finite) of bandwidth $N+1$. 

\item Equations of unbounded order are associated with nonsingular lower triangular matrices. 

\item Equations of ascending order are associated with row-finite matrices (but not in general column finite). This is the most complete form amongst RO-LDEVCs.
\end{itemize}
 
All RO-LDEVCs, are associated with coefficient matrices in lower echelon form (see section \ref{sec:DifferenceEquations}) with zero left-nullity. Therefore the infinite Gaussian elimination is solely implemented, as illustrated in example \ref{sec:Example1}. The algorithm results in a row-finite matrix $\QHF$, the unique Hermite form (HF) of $\A$ (or LRREF of $\A$).
\subsection{Linear Difference Equation with Variable Coefficients of\\ Ascending Order: General Solution Sequence}
\label{sec:TheNonhogeneousSolution}
The results of the previous section are applied hereby to construct the general solution sequence of a linear difference equation of ascending order. 

As $\A$ is in lower row echelon form, the rows of $\A$ are of strictly increasing length 
\[\mu_0=N, \mu_1=N+1,..., \mu_m=N+m,...,\]
and the set \emph{set of inaccessible row-lengths} of $\RS(\A)$ is $\clengths_{\A}=\{0,1,..,N-1\}$. By virtue of Eq. (\ref{deficiency of A}) the deficiency of $\A$ is $N$. The infinite Gaussian elimination gives the HF of $\A$:
\begin{equation} \label{QHF2}
\QHF=\left(\begin{array}{ccccccccc} 
h_{0,0} & h_{0,1} &...  &  h_{0, N-1}   &  1  &  0  &   0   &  0   &  ...  \\
h_{1,0} & h_{1,1} &...  &  h_{1, N-1}   &  0  &  1  &   0   &  0   & ...   \\
h_{2,0} & h_{2,1} &...  &  h_{2, N-1}   &  0  &  0  &   1   &  0   & ...   \vspace{-0.05in}\\
 .     & .      &...  &      .         &  .  &  .  &   .   &  .   & ...   \vspace{-0.1in}\\
 .     & .      &...  &      .         &  .  &  .  &   .   &  .   & ...   \vspace{-0.1in}\\ 
 .     & .      &...  &      .         &  .  &  .  &   .   &  .   & ...    
\end{array} \right).  \end{equation}
Notice that the matrix $\QHF$ in (\ref{QHF2}) is a special case of the matrix in (\ref{QHF}).

By applying the same sequence of elementary row operations to the rows of the identity matrix $\I$, as those employed for the reduction of $\A$ to $\QHF$, the resulting matrix $\Q$ is a lower triangular and nonsingular $\omega\times\omega$  matrix of the form:
\begin{equation} \label{Invertible Q}
\Q=\left(\begin{array}{ccccccccccc} 
q_{0,0} &  0      &   0   &  0   &  ...   \\
q_{1,0} &  q_{1,1} &   0   &  0   &  ... \\
q_{2,0} & q_{2,1}  &   q_{2,2}   &  0   & ...  \vspace{-0.05in}\\
 .     & .       &   .   &      & ...  \vspace{-0.1in}\\
  .     & .       &   .   &      & ... \vspace{-0.1in}\\
   .     & .       &   .   &      & ...        
\end{array} \right).  \end{equation}
Formally $\Q \cdot \A = \QHF$ and $\displaystyle q_{n,n}=\frac{1}{a_{n,N+n}}$.


According to Eq. (\ref{consistent}), on account of $W=\emptyset$, for every $g\in \complex^\infty$ the system (\ref{QLR}) is consistent. 

In view of $\QHF$ in (\ref{QHF2}), the solution sequence in (\ref{general solution 2}) takes the form
\begin{equation} \label{GSLDE}
y=(y_0,\! ...,y_{N-1}, \textbf{k}_0-\!\!\!\sum^{N-1}_{k=0}  h_{0k} y_k,  \textbf{k}_1-\!\!\!\sum^{N-1}_{k=0} h_{1k} y_k,...,\textbf{k}_i-\!\!\!\sum^{N-1}_{k=0} h_{ik} y_k,...)^T, 
\end{equation}
where $\textbf{k}_i=\Q_i\cdot g$ (recall that $\Q_i$ is the $i$ row of $\Q$) and $y_k$'s are free constants for all $k: 0\le k\le N-1$. 

The sequence in (\ref{GSLDE}) is the general solution sequence of the linear difference equation of ascending order of index $N$.
Accordingly, the initial value problem $y_k=c_k$, $0\le k\le N-1$, has a unique solution, as in the case of the $N$-order linear recurrence. 
The $y_{n+N}$ term of the solution sequence is given by
\[ y_{n+N}=\textbf{k}_n-\sum^{N-1}_{k=0} h_{nk} c_k \]
for all $n\ge 0$. 

Taking into account that  $\textbf{k}_n=(q_{n0},q_{n1},...,q_{nn})\cdot (g_0,g_1,...,g_n)^T$ we can also write:
\begin{equation} \label{unique solution 2}
y_{n+N}= \sum^{n}_{k=0} q_{nk} g_k -\sum^{N-1}_{k=0} h_{nk} c_k, \ \  n\ge 0.
\end{equation}

The infinite Gaussian elimination algorithm generates the chain of submatrices of $\QHF$ such that $n=\delta_n$ for all $n\in\naturals$, as in (\ref{Chain Order-2}), coupled with a chain of submatrices of $\Q$. Thus, the quantities $q_{nk}$ and $h_{nk}$ in (\ref{unique solution 2}) are fully determined at the $n$-step of the algorithm. 

The solution in (\ref{unique solution 2}) is expressed in terms of the $N$ initial conditions, the forcing terms $g_k$ and the quantities constructed by the algorithm. 
\subsection{Linear Difference Equation with Variable Coefficients of Order $N$}
\label{sec:NOrderLinearDifferenceEquationWithVariableCoefficients}
The $N$th order linear difference equation with variable coefficients is given by
\begin{equation} \label{NODE}
a_{n,n+N} y_{n+N}+a_{n,n+N-1} y_{n+N-1}+...+a_{n,n}y_{n}=g_n, \ \ n\in\naturals. 
\end{equation}
Equation (\ref{NODE}) takes the row and column finite system form:
\begin{equation} \label{MHODE} 
\left(\begin{array}{ccccccccc}
\!\!\! a_{00} &\!\!\!  a_{01} &\!\!\!\! ...  &\!\!\!  a_{0N-1} &  a_{0N}    &\!\!  0     &\!\!\!   0     &\!\!\!\!  0  &\!\!\! ... \!\!\! \\
\!\! 0      &\!\!\!  a_{11} &\!\!\!\!  ...  &\!\!\!  a_{1N-1}  &  a_{1N}     &\!\! a_{1N+1}   &\!\!\!   0     &\!\!\!\!  0  &\!\!\! ... \!\!\! \\
\!\! 0      &\!\!\!   0      &\!\!\!\!  ...  &\!\!\!  a_{2N-1} & a_{2N} &\!\! a_{2N+1}&\!\!\! a_{2N+2}  &\!\!\!\!  0  &\!\!\! ... \!\!\! \vspace{-0.05in}\\
\!\! .      &\!\!\!  .      &\!\!\!\!  ...  &\!\!\!  .      &   .        &\!\!  .       &\!\!\!   .     &\!\!\!\!  .  &\!\!\! ... \!\!\! \vspace{-0.1in}\\
\!\! .      &\!\!\!  .      &\!\!\!\!  ...  &\!\!\!  .      &   .        &\!\!  .       &\!\!\!   .     &\!\!\!\!  . &\!\!\! ... \!\!\! \vspace{-0.1in}\\
\!\! .      &\!\!\!  .      &\!\!\!\!  ... &\!\!\!  .      &   .        &\!\!  .       &\!\!\!   .     &\!\!\!\!  .  &\!\!\!... \!\!\! 
\end{array} \right) \!\!
 \left(\begin{array}{c}
 y_{0}  \\
 y_{1} \vspace{-0.05in}\\
  .   \vspace{-0.1in}\\
  .   \vspace{-0.1in}\\
  . \\
  y_{N-1}\\
  y_N \\
  y_{N+1} \\
  y_{N+2} \vspace{-0.05in}\\
  .   \vspace{-0.1in}\\
  .   \vspace{-0.1in}\\
  .   
 \end{array} \right)
\!\! =\!\! \left(\begin{array}{c}
 g_0 \\
 g_1 \\
 g_2 \\
  .   \vspace{-0.1in}\\
  .   \vspace{-0.1in}\\
  .  \end{array} \right).
\end{equation}

The results of subsection \ref{sec:TheNonhogeneousSolution} are directly applicable to Eq. (\ref{NODE}). 

Although the coefficient matrix of Eq. (\ref{NODE}) is row and column finite, its Hermite form is row-finite, given by (\ref{QHF2}). 
The general solution sequence of Eq. (\ref{NODE}) is given by (\ref{GSLDE}).
\section{Fundamental Set of Solutions: Extensions and Construction}
\label{sec:AFundamentalSetOfHomogeneousSolutions}
The fundamental set of solutions is a major tool for studying the global stability of dynamic systems. In the case of difference equations of $N$th order with variable coefficients, the fundamental solution set consists of $N$ linearly independent homogeneous solutions, thus forming a finite basis of the space of homogeneous solutions. 
\subsection{Linear Difference Equations of Regular Order}
\label{sec:LinearRecurrence}
In this subsection, the fundamental solution set of a RO-LDEVC is constructed. Setting $\textbf{k}_{0}=\textbf{k}_{1}=\textbf{k}_{2}=...=0$ in (\ref{GSLDE}), we get the general homogeneous solution of the ascending order linear difference equation of index $N$: 
\begin{equation} \label{HOGHS} \displaystyle(y_0,\! ...,y_{N-1}, -\!\!\!\sum^{N-1}_{k=0}  h_{0k} y_k,  -\!\!\!\sum^{N-1}_{k=0} h_{1k} y_k,..., -\!\!\!\sum^{N-1}_{k=0} h_{ik} y_k,...)^T.
\end{equation}
The homogeneous solution, $\mbox{\boldmath$\mbox{\boldmath$\xi$}$}^{(i)}$, is obtained by setting in (\ref{HOGHS}): $y_i=1$ and $y_j=0$ for all $j\in \{0,1,...,N-1\}$ with $j\not=i$. Thus,
\begin{equation} \label{FSHS}
\begin{array} {llccccl}
\mbox{\boldmath$\xi$}^{(0)}     &=(1,0,...,0, & -h_{00},  &-h_{1 0},  &... & -h_{n 0}, &...)^T \\
\mbox{\boldmath$\xi$}^{(1)}     &=(0,1,...,0, & -h_{01},  &-h_{1 1},  &... & -h_{n 1}, &...)^T \vspace{-0.05in}\\
.          &.            &.          &.          &... &  .        &... \vspace{-0.1in}\\
.          &.            &.          &.          &... &  .        &... \vspace{-0.1in}\\
.          &.            &.          &.          &... &  .        &... \\
\mbox{\boldmath$\xi$}^{(N-1)} &=(0,0,...,1, & -h_{0N-1},  &-h_{1 N-1},  &... & -h_{n N-1}, &...)^T .
\end{array} 
\end{equation}
In view of the matrix $\QHF$ in (\ref{QHF2}), the product $\QHF\cdot \mbox{\boldmath$\xi$}^{(i)}=\textbf{0}$ for all $0\le i\le N-1$, is also the direct result of $\QHF_k\cdot \mbox{\boldmath$\xi$}^{(i)}=0$ for all $k\in \naturals$. Notice that the term  $\mbox{\boldmath$\xi$}_{k}^{(i)}$ for $k\ge 0$ and $0\le i\le N-1$ in (\ref{FSHS}) is the opposite sign $(k,i)$, entry of $\QHF$, that is $\mbox{\boldmath$\xi$}_{k}^{(i)}=-h_{ki}$. 

Accordingly the solution sequences in (\ref{FSHS}) are primarily and simultaneously constructed by the infinite Gaussian elimination algorithm.
\begin{theorem} \label{fundamental solutions} The set $\mbox{\boldmath$\xi$}=\{\mbox{\boldmath$\xi$}^{(i)}\}_{0\le i\le N-1}$ is a fundamental solution set. 
\end{theorem}
\begin{proof}
In view of the sequence (\ref{HOGHS}), any $x\in\RNS(\A)$ can be written as:
\[x=\displaystyle(a_0,\! ...,a_{N-1}, -\!\!\!\sum^{N-1}_{k=0}  h_{0k} a_k,  -\!\!\!\sum^{N-1}_{k=0} h_{1k} a_k,..., -\!\!\!\sum^{N-1}_{k=0} h_{ik} a_k,...)^T =\displaystyle \sum^{N-1}_{k=0}\mbox{\boldmath$\xi$}^{(k)} a_k. \]
Thus, $\mbox{\boldmath$\xi$}$ spans $\RNS(\A)$. The linear independence follows from the fact that the Casoratian $W(0)$ of $\mbox{\boldmath$\xi$}$ is $W(0)=|\I|=1\not=0$. Therefore, $\mbox{\boldmath$\xi$}$ is a Hamel basis of $\RNS(\A)$.
\end{proof} 
Theorem \ref{fundamental solutions} entails the following statement:
\begin{corollary} \label{Extension of Fundamental Theorem} Let $N$ be the index of a linear difference equation of regular order. Then
\[\defic(\A)=\dim(\RNS(\A))=N.\]
\end{corollary}
This result extends the fundamental theorem of linear difference equations to cover the case of linear difference equations of ascending order with variable coefficients. 
\subsection{Row-Finite Systems}
\label{sec:LinearQuasi-Recurrencies}
In the present subsection, the notion of fundamental solution set is further extended to cover the general case of row-finite systems and therefore the case of linear difference equations of irregular order. The complete Gauss-Jordan elimination is applied to the coefficient matrix of (\ref{MQLR}) giving the QHF, $\QHF$, of $\A$. 
Let $x=(x_n)^T_{n\in\omega}$ be an arbitrary element of $\RNS(\A)$, that is a homogeneous solution of Eq. (\ref{MQLR}). According to the solution sequence in (\ref{homogeneous solution 2}), $x$ must be in the form: 
\begin{equation} \label{homogeneous solution 3} \begin{array}{llll}x=&(x_0, ...,x_{\mu_0-1}, -\!\!\displaystyle \sum^{\mu_0-1}_{k=0}\! h_{j_0k} x_k,&\hspace{-0.1in} x_{\mu_0+1}, ...,&\hspace{-0.15in} x_{\mu_1-1}, -\!\!\displaystyle\sum^{\mu_1-1}_{k=0} \! h_{j_ik} x_k,x_{\mu_1+1},...,\\
& &\hspace{-0.1in} x_{\mu_{i-1}+1},...,&\hspace{-0.1in}x_{\mu_i-1}, -\!\!\displaystyle\sum^{\mu_i-1}_{k=0} \! h_{j_ik} x_k,x_{\mu_i+1},\! ...)^T.
\end{array}\end{equation}
For each $i\in \clengths_{\A}$ define the sequence $\mbox{\boldmath$\xi$}^{(i)}$ in $\RNS(\A)$, by setting in Eq. (\ref{homogeneous solution 3}) $x_i=1$ and $x_j=0$ for all $j\in \clengths_{\A}$ with $j\not=i$, that is
\begin{equation} \label{sequence of fundamental solutions 1}
\begin{array} {lllll}\!\!\!\!\!\!
\mbox{\boldmath$\xi$}^{(i)}\!=\!(0,0,...,0,\!\!\!\!\!&1,0,...,0, \!\!\!\!\! & -h_{j_ni}, 0,...,0,\!\!\!\!\!&-h_{j_{n+1}i},0,...,0,\!\!\!\!\! &-h_{j_{n+m}i},0,...)^T  \\ 
& \uparrow           &\,\,\,\,\,\,\uparrow       &\,\,\,\,\,\,\uparrow        &\,\,\,\,\,\,\uparrow   \\
 & i                  & \,\,\,\,\,\,\mu_n        &\,\,\,\,\,\, \mu_{n+1}         &\,\,\,\,\,\, \mu_{n+m}
               \end{array}  
\end{equation}
whenever $\mu_{n-1}<i<\mu_n$ or $0\le i< \mu_0$. 
According to (\ref{sequence of fundamental solutions 1}), each term $h_{j_ni}$ in $\mbox{\boldmath$\xi$}^{(i)}$ is an entry of the $i$-column of $\QHF$, which belongs to the $j_n$ non-zero row of $\QHF$ and has the position $\mu_n$ in $\mbox{\boldmath$\xi$}^{(i)}$. 

For example, if $0\le i <\mu_0$, then:
\[\begin{array} {lllll}
\mbox{\boldmath$\xi$}^{(i)}=(0,0,...,0,\!\!&1,0,..., \!\!\!\!\! & 0,-h_{j_0i}, 0,...,0,\!\!\!\!\!&-h_{j_1i},0,...,0,\!\!\!\!\! &-h_{j_n i},0,...)^T . \\ 
& \uparrow  &\hspace{0.3in}\uparrow  &\hspace{0.2in}\uparrow &\hspace{0.2in} \uparrow   \\
& i                 & \hspace{0.3in}\mu_0        &\hspace{0.2in} \mu_{1}         &\hspace{0.2in} \mu_{n}
               \end{array} 
\]

Next, define the indexed set $\mbox{\boldmath$\xi$}=\{\mbox{\boldmath$\xi$}^{(i)}\}_{i\in \clengths_{\A}}$. We conclude that every sequence $\mbox{\boldmath$\xi$}^{(i)}$ in $\mbox{\boldmath$\xi$}$ is constructed by the infinite Gauss-Jordan elimination algorithm along with the construction of matrix $\QHF$. 
\begin{theorem} \label{Finite basis} If $\defic(\A)<\infty$, then the set $\mbox{\boldmath$\xi$}$ is a finite basis of the homogeneous solution space $\RNS(\A)$ of \emph{Eq. (\ref{MQLR})}.
\end{theorem}
\begin{proof} By virtue of proposition \ref{deficiency of A}, the hypothesis entails that the set $\clengths_{\A}$ is finite \[\clengths_{\A}\!=\!\{s_0,s_1,s_2,...,s_m\}\] 
such that $s_k<s_{k+1}$. Also $\mbox{\boldmath$\xi$}=\{\mbox{\boldmath$\xi$}^{(s_k)}\}_{0\le k\le m}$. 
If $x\in\RNS(\A)$, then it follows from (\ref{homogeneous solution 3}) that  $x=\sum_{k=0}^{m}\mbox{\boldmath$\xi$}^{(s_k)}x_{s_k}$\vspace{0.05in}, whence $\mbox{\boldmath$\xi$}$ spans $\RNS(\A)$.
It remains to show that $\mbox{\boldmath$\xi$}$ is linearly independent. The finite sequence $\beta_0,\beta_1,...,\beta_{s_m}$ in $\complex^{(\omega)}$ is defined as follows
\[\beta_{n}=\left\{\begin{array}{lll} \e^T_n      & {\rm if} \ n\not=s_k   \\                                    &          & 0\le n\le s_m, \\
\mbox{\boldmath$\xi$}^{(n)} & {\rm if} \ n=s_k \end{array}\right.  \]
where $\e^T_n$ is the transpose of $\e_n$ (see section \ref{sec:TheInfiniteGauss-JordanAlgorithmAndTheQuasi-HermiteForm}).

Call $\B$ the $\omega\times(s_m+1)$ matrix with columns $\beta_n$, $0\le n\le s_m$. Also call $\B_{s_m}$ the submatrix of $\B$ consisting of the first $s_{m+1}$ rows of $\B$. Evidently $\B_{s_m}$ is a nonsingular $(s_m+1)\times(s_m+1)$ matrix, since it is a lower triangular with non-zero diagonal entries. Thus the Casoratian $W(0)$ of $\B$ is non-zero. It follows that the columns of $\B$ are linearly independent. Thus, $\mbox{\boldmath$\xi$}$ is linearly independent, being a subset of a linearly independent set, and the assertion follows.
\end{proof}
Let $q=(q_0,q_1,q_2,...)^T$ and $p=(p_0,p_1,p_2,...)^T$ be elements of $\complex^\infty$. The sequence space $\complex^\infty$ equipped with the metric 
\begin{equation} \label{metric}
\varrho(q,p)=\sum_{i=0}^{\infty} \frac{1}{2^i}\frac{\left|q_i-p_i\right|}{1+\left|q_i-p_i\right|}
\end{equation}
is customarily denoted by $\mathfrak{s}$. The space $\mathfrak{s}$ is a locally convex, complete and metrizable (but not normable) topological vector space (Fr\'{e}chet space).
The canonical basis $(\e_n)_{n\in \omega}$ of $\complex^{(\omega)}$ is a Schauder (or countable) basis of $\mathfrak{s}$, which turns $\mathfrak{s}$ into a separable space. The convergence in $\mathfrak{s}$ is the coordinatewise convergence. 
\begin{theorem} \label{Schauder basis} If $\defic(\A)=\infty$, then the set $\mbox{\boldmath$\xi$}$ is a Schauder basis of the homogeneous solution space $\RNS(\A)$ of \emph{Eq. (\ref{MQLR})}.
\end{theorem}
\begin{proof} As $\card(\clengths_{\A})=\defic(\A)=\infty$, we write $\clengths_{\A}=\{s_0,s_1,s_2,...\}$ with $s_i<s_{i+1}$.  Also, $\mbox{\boldmath$\xi$}=\{\mbox{\boldmath$\xi$}^{(s_k)}\}_{k\in\omega}$. Let $x\in\RNS(\A)$. Then $x=(x_n)^T_{n\in\omega}$ is given by (\ref{homogeneous solution 3}). Next define the sequence $\phi=(\phi^{(n)})_{n\in\omega}$ by $\displaystyle\phi^{(n)}=\sum_{k=0}^{n}\mbox{\boldmath$\xi$}^{(s_k)}x_{s_k}$.
Let $m\in\omega$ such that $0\le m\le n$. Then, either $m=s_i\in \clengths_{\A}$ or $m=\mu_i\in\lengths_{\A}$ for some $i\in \omega$. If $m=s_i$, it follows from  (\ref{homogeneous solution 3}), that the $m$ coordinate of $\phi^{(n)}$ is $\phi^{(n)}_m=x_m$. If $m=\mu_i$, on account of $\mu_i=m\le n$, (\ref{homogeneous solution 3}) entails that the $m$ coordinate of $\phi^{(n)}$ is $\phi^{(n)}_m=-\sum^{\mu_i-1}_{k=0} h_{j_ik} x_k$. Thus, $\phi^{(n)}_m=x_m$ for all $m: 0\le m\le n$. Taking into account that $\phi_m^{(n)}-x_m=0$ for $m\le n$, we have:
\[\varrho(\phi^{(n)},x)=\sum_{i=n+1}^{\infty} \frac{1}{2^i}\frac{\left|\phi_i^{(n)}-x_i\right|}{1+\left|\phi_i^{(n)}-x_i\right|}<\sum_{i=n+1}^{\infty}\frac{1}{2^i}=\frac{1}{2^n}. \] 
It follows from $\varrho(\phi^{(n)},x)\longrightarrow 0$, as $n\longrightarrow \infty$, that 
each $x\in\RNS(\A)$ is decomposable as a series $\displaystyle x=\sum_{k=0}^{\infty}\mbox{\boldmath$\xi$}^{(s_k)}x_{s_k}$. 

It remains to show that the sequence $(x_{s_k})$ is unique. For this purpose, assume that $\displaystyle x=\sum_{k=0}^{\infty}\mbox{\boldmath$\xi$}^{(s_k)}z_{s_k}$. The $n$th term of the sequence of partial sums is  $\displaystyle\tau^{(n)}=\sum_{k=0}^{n}\mbox{\boldmath$\xi$}^{(s_k)}z_{s_k}$. The sequence $(\tau^{(n)})_{n\in\omega}$ converges to $x$, that is $\varrho(\tau^{(n)},x)\longrightarrow 0$, as $n\longrightarrow \infty$. Define the sequence $z=(z_n)_{n\in\omega}$ as follows:
\[z_n=\left\{\begin{array}{ccl} z_{s_k}  & {\rm if} & n=s_k \vspace{0.1in}\\
 \displaystyle -\sum^{\mu_k-1}_{m=0} h_{j_km} z_m  &{\rm if}  & n=\mu_k.
                   \end{array}\right.
\]
In view of Eq. (\ref{homogeneous solution 2}), it follows that $z\in\RNS(\A)$. Repeating the above arguments we infer that $\varrho(\tau^{(n)},z)\longrightarrow 0$, as $n\longrightarrow \infty$. The uniqueness of limits of convergent sequences in metric spaces entails that $x=z$, whence $x_{s_k}=z_{s_k}$ for all $k\in\omega$, as required. 
\end{proof}

\subsection{Classification of Row-Finite Systems and a Unified Form of their General Homogeneous Solution}
\label{sec:ClassificationOfRowFiniteSystems}
The Hamel basis of the homogeneous solution space of RO-LDEVCs is characterised as fundamental solution set. Theorem 5 shows that this characterisation can be directly extended to row-finite systems with finite deficiency. By virtue of theorem \ref{Schauder basis}, the homogeneous solution space of row-finite systems with infinite deficiency has a Schauder basis. The notion of Schauder basis coincides with that of Hamel basis in the finite dimensional case. Therefore, the basis $\mbox{\boldmath$\xi$}$ of the homogeneous solution space obtained by the infinite Gauss-Jordan algorithm can be generally characterised as \emph{fundamental solution set}.

To summarize, row-finite linear systems can be classified according to the deficiency of their coefficient matrix as follows:
\begin{enumerate}
\item When the deficiency is finite, they yield a finite dimensional space of homogeneous solutions that has a Hamel basis. They include all RO-LDEVCs and all LDEVCs of irregular order with finite index. 
\item When the deficiency is infinite they yield an infinite-dimensional space of homogeneous solutions that has a Schauder basis. They include all LDEVCs of irregular order with infinite index. \end{enumerate}




The general homogeneous solution $y=(y_k)_{k\in\omega}$ of any arbitrary row-finite $\omega\times \omega$ linear system ($\defic(\A)<\infty$ or $\defic(\A)=\infty$) is given by the unified formula
\begin{equation}\label{GGHS}
\displaystyle y=\sum_{k=0}^{\defic(\A)}c_{k}\ \mbox{\boldmath$\xi$}^{(s_k)},
\end{equation}
where $\{s_k\}_{k=0,1,...,\defic(\A)}$ is the set $\clengths_{\A}$ of inaccessible row-lengths of $\RS(\A)$, $\{\mbox{\boldmath$\xi$}^{(s_k)}\}_{k=0,1,...,\defic(\A)}$ is the set of fundamental solutions and $c_{k}$ are free constants for all $k=0,1,...,\defic(\A)$.
\section{Solutions for Linear Difference Equations \\of Regular Order in terms of Hessenbergians}
\label{sec:ExplicitSolutionsForLinearDifferenceEquationsOfRegularOrder}
Up to now, the infinite Gaussian elimination has been used as an algorithm to construct solution sequences for row-finite systems. 
In this section the infinite Gaussian elimination will serve to provide, through symbolic computation, an explicit solution formula for RO-LDEVCs (starting from lower order LDEVCs) in terms of Hessenbergians. In theorem \ref{th. AO(N)}, these formulas will be generalised to cover all RO-LDEVCs applying the mathematical induction.
\subsection{First Order Linear Difference Equation}
\label{sec:FirstOrderLinearDifferenceEquation}
The standard solution (see ~\cite{El:Int}) of the initial value problem
\begin{equation} \label{first order} \left\{ \begin{array}{ll}y_n&=a_ny_{n-1} ,\ \   n\in\naturals\\
                                                                y_{-1}&=c_0\end{array}\right.,
\end{equation}
with $a_n\not=0$ for some $n\in\naturals$, is recovered by applying the infinite Gaussian elimination with right pivoting to the coefficient matrix 
\begin{equation} \label{FODE}
\A=\left(\begin{array}{ccccc}
 -a_0  &  1        &   0      &   0    &\!\!\! ... \!\!\! \\
  0    & -a_1      &   1      &   0    &\!\!\! ... \!\!\! \\
  0    &   0        & -a_2    &   1 
  &\!\!\! ... \!\!\! \vspace{-0.05in}\\   
  \vdots & \vdots &\vdots &\vdots 
\end{array} \right)
\end{equation}
of Eq. (\ref{first order}). It results in the Hermite Form of $\A$:
\[\QHF=\left(
\begin{array}{ccccccc}
-a_0        & 1 & 0 & 0 & 0 &...\\
-a_0 a_1   & 0 & 1 & 0 & 0 &...\\
-a_0 a_1a_2 & 0 & 0 & 1 & 0 &...\vspace{-0.05in}\\
   .&.&.&.&.&...\vspace{-0.1in}\\
   .&.&.&.&.&...\vspace{-0.1in}\\
   .&.&.&.&.&...
\end{array}
\right).\]
In view of (\ref{FSHS}), the fundamental solution $\mbox{\boldmath$\xi$}$ of Eq. (\ref{first order}) is formulated in terms of the opposite-sign, first column of $\QHF$:
\[\mbox{\boldmath$\xi$}=(1,a_0,a_0 a_1,a_0 a_1a_2,...,\prod_{i=0}^na_i,...)^T.\]
Let us call $\mbox{\boldmath$\xi$}(n)$ the general term of $\mbox{\boldmath$\xi$}$. Then,
\[y_n=c_0\mbox{\boldmath$\xi$}(n)=[\prod_{i=0}^{n}a_i]c_0,\]
in full accord with the well known solution of the initial value problem of Eq. (\ref{first order}).
\subsection{Homogeneous Linear Difference Equation \\of Second Order}
\label{sec:ExplicitFundamentalSolutions}
The solution of the second order linear difference equation, with non-constant coefficients is a decisive step for extending the first order LDE solution to the $ N $-th and ascending order linear recurrence.

The normal form of the second order linear recurrence is given by
\begin{equation} \label{SVDE}
y_{n}+b_n y_{n-1}+a_n y_{n-2}=0, \hspace{0.2in} n\in\naturals,
\end{equation}
where $a_n\not=0$ for some $n\in\naturals$. The associated row-finite matrix $\A$ of equation (\ref{SVDE}) is given by (\ref{example1}). 

Following the results of example \ref{sec:Example1}, the infinite Gaussian elimination algorithm, implemented under a rightmost pivot strategy reduces $\A$ to its HF:
\begin{equation}\label{HF3}
\QHF=\left(\begin{array}{ccccccc}
a_0 & b_0 & 1 & 0 & 0 & 0 &...\\
 -a_0 b_1 & a_1-b_0 b_1 & 0 & 1 & 0 & 0 &...\\
 a_0 b_1 b_2-a_0 a_2 & -a_2 b_0+b_1 b_2 b_0-a_1 b_2 & 0 & 0 & 1 & 0 &...\vspace{-0.05in}\\
   .&.&.&.&.&.&...\vspace{-0.1in}\\
     .&.&.&.&.&.&...\vspace{-0.1in}\\
       .&.&.&.&.&.&...
\end{array}
\right)
\end{equation}
In view of (\ref{FSHS}) the first two opposite-sing columns of $\QHF$ augmented at the top by $(1,0)$ and $(0,1)$ respectively, form a fundamental solution set of Eq. (\ref{SVDE}):
\[\begin{array}{ll}
\mbox{\boldmath$\xi$}^{(1)}=&(1,0,-a_0, a_0 b_1,-a_0 b_1b_2\!+\!a_0 a_1,...)^T\\
\mbox{\boldmath$\xi$}^{(2)}=& (0,1,-b_0,-a_1\!+\!b_0 b_1,-b_0b_1b_2\!+\!a_1 b_1\!+\!a_1 b_0,...)^T.
\end{array}\]
We observe that the terms of theses fundamental solution sequences are expansions of certain determinants given by
\[ \begin{array}{ll} \mbox{\boldmath$\xi$}^{(0)}\!\!\!\!&=\left\{\begin{array}{ll} 
 \mbox{\boldmath$\xi$}_{-2 }^{(0)}=1  &  \vspace{0.03in}\\
\mbox{\boldmath$\xi$}_{-1 }^{(0)}=0  &   \vspace{0.03in}\\
\mbox{\boldmath$\xi$}_{0 }^{(0)}=-a_0 & \\ 
\mbox{\boldmath$\xi$}_{1 }^{(0)}=\det\left(\begin{array}{cc} 
  a_0 &  1  \\ 
  0   & b_1 
\end{array} \right) &\vspace{0.1in}\\
\mbox{\boldmath$\xi$}_{2}^{(0)}\!\!=\!-\!\det\left(\begin{array}{ccc} 
 a_0 &  1  & 0\\ 
 0   & b_1 & 1\\
 0   & a_2 & b_2 \end{array} \right)\vspace{-0.05in}\\
 . \vspace{-0.1in}\\
 .\vspace{-0.1in}\\
 .\end{array}\right.\end{array}\hspace{-0.35in} \begin{array}{ll}\mbox{\boldmath$\xi$}^{(1)}\!\!\!\!&=\left\{\begin{array}{ll}          \mbox{\boldmath$\xi$}_{-2}^{(1)}=0 \vspace{0.03in}\\  \mbox{\boldmath$\xi$}_{-1}^{(1)}=1 \vspace{0.03in}\\
 \mbox{\boldmath$\xi$}_{0}^{(1)}=-b_0  \\
 \mbox{\boldmath$\xi$}_{1}^{(1)}=\det\left(\begin{array}{cc} b_0 &  1  \\ 
    a_1 & b_1 \end{array} \right) & \vspace{0.1in}\\
 \mbox{\boldmath$\xi$}_{2}^{(1)}\!\!=\!-\!\det\left(\begin{array}{ccc} 
b_0 &  1  & 0\\ 
a_1 & b_1 & 1\\
0  & a_2 & b_2\end{array}\right)\vspace{-0.05in}\\ . \vspace{-0.1in}\\
.\vspace{-0.1in}\\
.\end{array}\right.\end{array} \]
This sequence leads to the subsequent general terms of $\mbox{\boldmath$\xi$}^{(0)},\mbox{\boldmath$\xi$}^{(1)}$:
\[ \mbox{\boldmath$\xi$}_{n}^{(0)}=(-1)^{n+1} \det\left(\begin{array}{ccccccc} 
a_0 &  1  & 0 & ...&0 &0 &0  \\ 
 0 &  b_1 & 1 & ...&0 &0 &0  \vspace{-0.05in}\\
 .   &.    &.  & ...&. &. &.  \vspace{-0.1in}\\ 
 .   &.    &.  & ...&. &. &.  \vspace{-0.1in}\\ 
  .   &.    &.  & ...&. &. &.  \\
  0   &  0  & 0 & ...&a_{n-1}&b_{n-1}& 1 \\
  0   &  0  & 0 & ...&   0   &a_{n}&b_{n} 
  \end{array} \right),  \ n\ge 0,\]                               
\[ \mbox{\boldmath$\xi$}_{n}^{(1)}=(-1)^{n+1} \det\left(\begin{array}{ccccccc} 
 b_0 &  1  & 0 & ...&0 &0 &0  \\ 
 a_1 & b_1 & 1 & ...&0 &0 &0 \vspace{-0.05in}\\
 .   &.     &.  & ...&. &. &. \vspace{-0.1in}\\
 .   &.     &.  & ...&. &. &. \vspace{-0.1in}\\
 .   &.     &.  & ...&. &. &. \\
 0   &  0  & 0 & ...&a_{n-1}&b_{n-1}  &  1 \\
 0   &  0  & 0 & ...&  0    &  a_{n}& b_{n}
\end{array} \right),  \   n\ge 0.  \]
 A formal proof of these formulas in a more general form will follow in the next subsection  (see theorem \ref{th. AO(N)}).
\subsection{Linear Difference Equation of Ascending Order} 
Let us start by considering the linear recurrence of ascending order with variable coefficients having index $1$:
\begin{equation} \label{HAO1}
y_{n}+a_{n,n} y_{n-1}+...+a_{n,1} y_{0}+a_{n,0}y_{-1}=0.
\end{equation}
The associated matrix is \vspace{-0.03in}
\begin{equation} \label{MAO1}\!\!\!\!\!  
\A=\left(\begin{array}{ccccccc}
a_{0,0} &   1    & 0 &  0 &  0  &  ...  \\
a_{1,0} & a_{1,1} & 1 &  0 &  0  &   ...  \\
a_{2,0} & a_{2,1} &a_{2,2} & 1 &  0  &   ...  \vspace{-0.03in}\\
\vdots & \vdots & \vdots     & \vdots    &  \vdots    & \vdots\vdots\vdots 
\end{array} \right).
\end{equation}
The infinite Gaussian elimination  gives the HF of $\A$:
\begin{equation} \label{HAO1}  
\QHF=\left(\begin{array}{ccccccc}
a_{0,0} &   1    & 0 &  0 &  0  &  ...  \\
a_{1,0}-a_{0,0}a_{1,1} & 0 & 1 &  0 &  0  &   ...  \\
a_{2,0}-a_{0,0}a_{2,1}-a_{2,2}a_{1,0}+a_{0,0}a_{1,1}a_{2,2} & 0 & 0 & 1 &  0  &   ...  \vspace{-0.03in}\\
\vdots & \vdots & \vdots     & \vdots    &  \vdots    & \vdots\vdots\vdots 
\end{array} \right).
\end{equation}
The fundamental solution sequence consists of the opposite-sign entries of the first column of $\QHF$ in (\ref{HAO1}), augmented on their left by $1$s. These are expansions of full Hessenbergians as described below:\vspace{-0.1in}
\begin{equation}\label{fundamental solution AO(1)}
\mbox{\boldmath$\xi$}=\{1,-a_{0,0},\left|\begin{array}{cc} a_{0,0} & 1 \\ a_{1,0} & a_{1,1}\end{array}\right|, -\left|\begin{array}{ccc} a_{0,0} & 1 &0 \\ a_{1,0} & a_{1,1} & 1\\ a_{2,0} & a_{2,1} & a_{2,2} \end{array}\right|,... \}
\end{equation}

The fundamental solution in (\ref{fundamental solution AO(1)}), in addition to the fundamental solutions formulated in the previous subsections \ref{sec:FirstOrderLinearDifferenceEquation} and \ref{sec:ExplicitFundamentalSolutions} are all special cases of fundamental solutions for the homogeneous RO-LDEVC of index $N$, that is
\begin{equation} \label{HAON}
y_{n}+a_{n,N+n-1} y_{n-1}+...+a_{n,1} y_{1-N}+a_{n,0}y_{-N}=0, \ \ \ n\in\naturals.
\end{equation}
\begin{theorem} \label{th. AO(N)} The general term of the fundamental solution sequence $\mbox{\boldmath$\xi$}^{(i)}$, $0\le i\le N-1$, of \emph{Eq. (\ref{HAON})}, with initial condition values   
\begin{equation}\label{initial fundamental conditions}
\mbox{\boldmath$\xi$}_{n}^{(i)}=\left\{\begin{array} {cl} 0, & {\rm if} \ n\not=i-N \\  
1, & {\rm if} \ n=i-N    \end{array}\right.\!\!\!\! , \ \ {\rm for} \ -N\le n\le -1, \end{equation}
is of the form:
\begin{equation} \label{fundamental i}\begin{array}{l}\mbox{\boldmath$\xi$}_{n}^{(i)}=\\
(-1)^{n+1} \det\left(\!\!\!\begin{array}{cccccc} 
 a_{0i} \!\!\!\!& 1 \!\!\!\!  & 0 \!\!\!\!\!& ...\!\!\!\! &0\!\!\! &0  \\ 
 a_{1i} \!\!\!\!&  a_{1N}\!\!\! & 1 \!\!\!\! & ...\!\!\!\!&0\!\!\! &0  \vspace{-0.06in}\\
 \vdots  \!\!\!\!&\vdots    \!\!\!\!&\vdots  \!\!\!\!& \vdots\vdots\vdots\!\!\!\!&\vdots\!\!\! &\vdots \vspace{-0.04in}\\ a_{n-1,i}\!\!\!\!   &a_{n-1,N}\!\!\!\!  & a_{n-1,N+1}\!\!\!\! & ...\!\!\!\!&a_{n-1, n+N-2}\!\!\!& 1\\
 a_{n,i}\!\!\!\! &a_{n,N}\!\!\!\!  & a_{n,N+1}\!\!\!\!  & ...\!\!\!\!&a_{n, n+N-2}\!\!\!& a_{n, n+N-1}
 \end{array}\!\!\! \right)\end{array}
  \end{equation}
for $n\ge 0$.                               
\end{theorem}
\begin{proof}
The proof is by induction on $n\in\naturals$. The unique solution of Eq. (\ref{HAON}), subject to the initial conditions in (\ref{initial fundamental conditions}), is the fundamental solution sequence $\mbox{\boldmath$\xi$}^{(i)}$ given by (\ref{FSHS}). The sequence
\[ y_{-N}=0,...,y_{i-N}=1,...,y_{-1}=0, y_{0}=-a_{0i}\]
solves the equation (\ref{HAON}) for $n=0$, that is $\mbox{\boldmath$\xi$}_{0}^{(i)}=-a_{0i}$ (which agrees with the solution (\ref{fundamental i}) for $n=0$). 

Next we adopt the notation:
\[ \mbox{\boldmath$\xi$}_{-N}^{(i)}=y_{-N}=0,...,\mbox{\boldmath$\xi$}_{i-N}^{(i)}=y_{i-N}=1,...,\mbox{\boldmath$\xi$}_{-1}^{(i)}=y_{-1}=0.\]
Let us call 
\[\begin{array}{c} D_k=\det\left(\begin{array}{cccccc} 
  a_{0i} &  1  & 0 & ...&0 &0  \\ 
  a_{1i} &  a_{1N} & 1 & ...&0 &0  \vspace{-0.05in}\\
  .   &.    &.  & ...&. &. \vspace{-0.1in}\\ 
  .   &.    &.  & ...&. &.   \vspace{-0.1in}\\ 
  .   &.    &.  & ...&. &.   \\
  a_{k-1,i}&a_{k-1,N}& a_{k-1,N+1}& ...&a_{k-1, k+N-2}& 1\\ 
  a_{k,i}   &a_{k,N}  & a_{k,N+1}  & ...&a_{k, k+N-2}& a_{k, k+N-1}
  \end{array} \right)\end{array}\]
The induction hypothesis states that $\mbox{\boldmath$\xi$}_{n}^{(i)}=(-1)^{n+1}D_n$ for all $n$: $1\le n\le k-1$,
that is $(-1)^{n+1}D_n$ solves  equation (\ref{HAON}) for $n\le k-1$.
Expanding $D_k$ along the last row we have:   
 \[\begin{array}{l}               D_k=(-1)^{2k+2} a_{k, k+N-1} \det\left(\!\!\begin{array}{ccccc} 
  a_{0i} &  1 \!\!\!\! &...\!\!\!\!&0 &0  \\ 
  a_{1i} &  a_{1N}\!\!\!\! & ...\!\!\!\!&0 &0  \vspace{-0.05in}\\
  .   &.  \!\!\!\!  & ...\!\!\!\!&. &. \vspace{-0.1in}\\ 
  .   &.   \!\!\!\! & ...\!\!\!\!&. &.   \vspace{-0.1in}\\ 
  .   &.   \!\!\!\!  & ...\!\!\!\!&. &.   \\
  a_{k-2,i}&a_{k-2,N}\!\!\!\!&  ...\!\!\!\!&a_{k-2, k+N-3}& 1\\ 
  a_{k-1,i}   &a_{k-1,N}\!\!\!\!  &  ...\!\!\!\!&a_{k-1, k+N-3}& a_{k-1, k+N-2}
 \end{array}\!\! \right)+\vspace{0.1in}\\
 (-1)^{2k+1} a_{k, k+N-2}\det\left(\begin{array}{ccccc} 
  a_{0i} &  1 \!\! & ...\!\!&0 &0  \\ 
  a_{1i} &  a_{1N}\!\! &...\!\!&0 &0  \vspace{-0.05in}\\
  .   &.   \!\! & ...\!\!&. &. \vspace{-0.1in}\\ 
  .   &. \!\!   &...\!\!&. &.   \vspace{-0.1in}\\ 
  .   &.  \!\!  &...\!\!&. &.   \\
  a_{k-3,i}&a_{k-3,N}\!\!& ...\!\!& a_{k-3, k+N-4}& 1\\ 
  a_{k-2,i}   &a_{k-2,N} \!\! &  ...\!\!&a_{k-2, k+N-4}& a_{k-2, k+N-3}
 \end{array} \right)+...\vspace{0.1in}\\
  ...+(-1)^{k+3} a_{k,N}a_{0,i}+(-1)^{k+2}a_{k,i}\cdot 1.
 \end{array}\]
Using the induction hypothesis and taking into account that $\mbox{\boldmath$\xi$}_{i-N}^{(i)}=1$, we have                            
 \[\begin{array}{lll}D_k&=&(-1)^{2k+2} a_{k, k+N-1} (-1)^k\mbox{\boldmath$\xi$}_{k-1}^{(i)}+\\
 &&(-1)^{2k+1} a_{k, k+N-2}(-1)^{k-1}\mbox{\boldmath$\xi$}_{k-2}^{(i)}+...+\\
 &&(-1)^{k+3} a_{k,N}(-1)^1 \mbox{\boldmath$\xi$}_{0}^{(i)}+(-1)^{k+2}a_{k,i}\mbox{\boldmath$\xi$}_{i-N}^{(i)},\end{array}\] 
whence
  \[\begin{array}{lll}
D_k &=&(-1)^{3k+2} a_{k, k+N-1} \mbox{\boldmath$\xi$}_{k-1}^{(i)}+\\
  &&(-1)^{3k} a_{k, k+N-2}\mbox{\boldmath$\xi$}_{k-2}^{(i)}+...+\\
  &&(-1)^{k+4} a_{k,N} \mbox{\boldmath$\xi$}_{0}^{(i)}+(-1)^{k+2}a_{k,i}\mbox{\boldmath$\xi$}_{i-N}^{(i)} \end{array}\]  
Multiplying both sides of the latter expression of $D_k$ by $(-1)^{k+1}$ we obtain:                                                      \[(-1)^{k+1}D_k=-a_{k, k+N-1}\mbox{\boldmath$\xi$}_{k-1}^{(i)}-a_{k, k+N-2}\mbox{\boldmath$\xi$}_{k-2}^{(i)}-...-a_{k,N}           \mbox{\boldmath$\xi$}_{0}^{(i)}-a_{k,i}\mbox{\boldmath$\xi$}_{i-N}^{(i)}.\]   
Transferring all the terms of the above equation to the left-hand side (in order to get zero on the right-hand side) and then adding the zero sum $\displaystyle\sum_{\substack{m=1\\ m\not=N-i}}^{N}\!\!a_{k,N-m}\mbox{\boldmath$\xi$}_{-m}^{(i)}$ (where each individual term is zero), we get:
\[\begin{array}{l}(-1)^{k+1}D_k+a_{k, k+N-1}\mbox{\boldmath$\xi$}_{k-1}^{(i)}+a_{k, k+N-2}\mbox{\boldmath$\xi$}_{k-2}^{(i)}+...+a_{k,N}\mbox{\boldmath$\xi$}_{0}^{(i)}\\
+a_{k,N-1}\mbox{\boldmath$\xi$}_{-1}^{(i)}+...+a_{k,i}\mbox{\boldmath$\xi$}_{i-N}^{(i)}+...+a_{k,0}\mbox{\boldmath$\xi$}_{-N}^{(i)}=0.\end{array}\]         Thus, the sequence
\[\begin{array}{ll}y_{-N}=\!\!\!&\mbox{\boldmath$\xi$}_{-N}^{(i)},...,\ y_{i-N}=\mbox{\boldmath$\xi$}_{i-N}^{(i)},...,\ y_{-1}=\mbox{\boldmath$\xi$}_{-1}^{(i)},\ y_{0}=\mbox{\boldmath$\xi$}_{0}^{(i)},...,\\
y_{k-1}=\!\!\!&\!\mbox{\boldmath$\xi$}_{k-1}^{(i)},\ y_k=(-1)^{k+1}D_k\end{array}\]
solves the equation (\ref{HAON}) for $n=k$, whence 
 \[\mbox{\boldmath$\xi$}_{k}^{(i)}=(-1)^{n+1}D_k\]
and the induction is complete.\end{proof}
The fundamental solution matrices associated with $\mbox{\boldmath$\xi$}_{k}^{(i)}$, $k\ge 0$, are in lower Hessenberg form.
On account of
\[(-1)^{n+1}\det\left(\begin{array}{ccccccc} -a_0  &  1   & 0 & ...&0 &0 &0  \\ 
  0   & -a_1 & 1 & ...&0 &0 &0  \vspace{-0.05in}\\
   .   &.     &.  & ...&. &. &. \vspace{-0.1in}\\
  .   &.     &.  & ...&. &. &. \vspace{-0.1in}\\
  .   &.     &.  & ...&. &. &. \\
  0   &  0   & 0 & ...& 0& -a_{n-1}  &  1 \\
  0   &  0   & 0 & ...&  0    &  0& -a_{n}
\end{array} \right)=\prod_{i=0}^na_i,\]
we conclude that the fundamental solution formula in (\ref{fundamental i}) naturally generalizes the standard solution of the first order linear recurrence (see subsection \ref{sec:FirstOrderLinearDifferenceEquation}). 
\subsection{The General Solution as a Single Hessenbergian}
\label{sec:ExplicitGeneralSolution}
Consider the normal form of the non-homogeneous RO-LDEVC:
\begin{equation} \label{NHAON}
y_{n}+a_{n,N+n-1} y_{n-1}+...+a_{n,1} y_{1-N}+a_{n,0}y_{-N}=g_n,\  \ n\in\naturals.
\end{equation}
A particular solution sequence $P=(p_{-N},p_{1-N},...,p_{-1},p_0,p_1,...,p_n,...)^T$ of Eq. (\ref{NHAON})  is given by
$p_{-N}=p_{1-N}=...=p_{-1}=0$ and 
\begin{equation} \label{particular solution 2}p_{n}=
(-1)^{n} \det\left(\!\!\!\begin{array}{cccccc} 
          g_{0} \!\!\!\!& 1\!\!\!\!  & 0 \!\!\!\!\!& ...\!\!\!\!&0\!\!\! &0  \\ 
          g_{1} \!\!\!\!&  a_{1N}\!\!\!\! & 1\!\!\! & ...\!\!\!\!&0\!\!\! &0  \vspace{-0.05in}\\
           .   \!\!\!\!&.    \!\!\!\!&.  \!\!\!\! & ...\!\!\!\!&.\!\!\! &. \vspace{-0.1in}\\ 
           .   \!\!\!\!&.    \!\!\!\!&.  \!\!\!\!& ...\!\!\!\!&.\!\!\! &. \vspace{-0.1in}\\
           .   \!\!\!\!&.    \!\!\!\!&.  \!\!\!\!& ...\!\!\!\!&. \!\!\!&. \\
       g_{n-1}\!\!\!\!   &a_{n-1,N}\!\!\!\!  & a_{n-1,N+1}\!\!\!\! & ...\!\!\!\!&a_{n-1, n+N-2}\!\!\!& 1\\
       g_{n}\!\!\!\! &a_{n,N}\!\!\!\!  & a_{n,N+1}\!\!\!\!  & ...\!\!\!\!&a_{n, n+N-2}\!\!\!& a_{n, n+N-1}
       \end{array}\!\!\! \right)
   \end{equation}
for $n\ge 0$. The proof follows the pattern of the theorem \ref{th. AO(N)} proof. 
The $n$th term of the general solution of Eq. (\ref{NHAON}) is given by:
\begin{equation} \label{unique solution 4} 
y_n=p_n+\sum_{i=0}^{N-1}\mbox{\boldmath$\xi$}_{n}^{(i)}y_{i-N}.
\end{equation}
In the following theorem we adhere to the subsequent notation: Any $(n+1)\times n$ matrix $\B$ augmented on the left by a $(n+1)\times 1$ column $c$ results in a square matrix  $(n+1)\times (n+1)$ denoted as  $[c\ ; \B]$.
\begin{theorem} \label{compact solution form} The general solution of the non-homogeneous RO-LDEVC (given in Eq. \emph{(\ref{NHAON})}) takes the single Hessenbergian form\emph{:}
\begin{equation} \label{compact general solution form}\begin{array}{l} y_{n}=\\
(-1)^{n}\! \det\!\left(\!\!\!\begin{array}{cccccc} 
 \displaystyle g_0\!-\!\!\sum_{i=0}^{N-1}a_{0i}y_{i-N} \!\!\!\!& 1\!\!\!\!  & 0 \!\!\!\!\!& ...\!\!\!\!&0\!\!\! &0  \\ 
\displaystyle g_1\!-\!\!\sum_{i=0}^{N-1}a_{1i}y_{i-N} \!\!\!\!&  a_{1N}\!\!\!\! & 1\!\!\! & ...\!\!\!\!&0\!\!\! &0  \vspace{-0.05in}\\
           .   \!\!\!\!&.    \!\!\!\!&.  \!\!\!\! & ...\!\!\!\!&.\!\!\! &. \vspace{-0.1in}\\ 
           .   \!\!\!\!&.    \!\!\!\!&.  \!\!\!\!& ...\!\!\!\!&.\!\!\! &. \vspace{-0.1in}\\
           .   \!\!\!\!&.    \!\!\!\!&.  \!\!\!\!& ...\!\!\!\!&. \!\!\!&. \\
\displaystyle g_{n-1}\!-\!\!\sum_{i=0}^{N-1}a_{n-1,i}y_{i-N}\!\!\!\!   &a_{n-1,N}\!\!\!\!  & a_{n-1,N+1}\!\!\!\! & ...\!\!\!\!&a_{n-1, n+N-2}\!\!\!& 1\\
\displaystyle g_n\!-\!\!\sum_{i=0}^{N-1}a_{n,i}y_{i-N}\!\!\!\! &a_{n,N}\!\!\!\!  & a_{n,N+1}\!\!\!\!  & ...\!\!\!\!&a_{n, n+N-2}\!\!\!& a_{n, n+N-1}
 \end{array}\!\!\! \right)\end{array}
\end{equation}
\end{theorem}
\begin{proof}
Matrices (\ref{particular solution 2}) and (\ref{fundamental i}) have in common the $(n+1)\times n$ submatrix:
\[ \mbox{\boldmath$\rm C$}_{(n+1)\times n}=\left(\begin{array}{cccc}
  1    &\!\!\!\!  0         &\!\!\!\!\!  ... \!\!\!\!\! & 0             \\
  a_{1N}    & \!\!\!\! 1  &\!\!\!\!\!  ... \!\!\!\!\! & 0             \\
    .         & \!\!\!\!  .        &\!\!\!\!\!  ... \!\!\!\!\! & .            \vspace{-0.13in} \\
   .         & \!\!\!\!  .        &\!\!\!\!\!  ... \!\!\!\!\! & .            \vspace{-0.13in} \\
   .         & \!\!\!\!  .        &\!\!\!\!\!  ... \!\!\!\!\! & .             \\
  a_{n-1N}  & \!\!\!\! a_{n-1\; N+1} &\!\!\!\!\!  ... \!\!\!\!\! & 1       \\ 
  a_{nN}    & \!\!\!\! a_{n\; N+1}   &\!\!\!\!\!  ... \!\!\!\!\! & a_{n\; n+N-1} \end{array}\right).
\]
The general solution $y_n$ in (\ref{unique solution 4}) can be written as:
\[y_n\!\!=\!\!(-1)^n\det\!\left[\!\!\left(\!\!\begin{array}{l}
 g_0\vspace{0.04in}\\
 g_1\\
 .  \vspace{-0.13in}\\
 . \vspace{-0.13in} \\
 .\\
 g_n\end{array}\!\!\!\right)\!;\! \mbox{\boldmath$\rm C$}_{(n+1)\times n}\!\right]\!\!+\!(-1)^{n+1}\!\sum_{i=0}^{N-1}y_{i-N}\det\!\left[\!\!\left(\!\!\begin{array}{l}
\displaystyle a_{0i} \vspace{0.04in}\\
\displaystyle a_{1i}\\
 .  \vspace{-0.13in} \\
 . \vspace{-0.13in} \\
 .\\
\displaystyle a_{ni} \end{array}\!\!\!\right)\!;\!\mbox{\boldmath$\rm C$}_{(n+1)\times n}\!\right].
\]
After suitable factorization, the above expression of $y_n$ takes the form:
\[
y_n=(-1)^n\ \left(\det\!\left[\!\!\left(\!\!\begin{array}{l}
 g_0\vspace{0.04in}\\
 g_1\\
 .  \vspace{-0.13in}\\
 . \vspace{-0.13in} \\
 .\\
 g_n\end{array}\!\!\!\right)\!;\! \mbox{\boldmath$\rm C$}_{(n+1)\times n}\!\right]\!\!-\!\displaystyle\sum_{i=0}^{N-1}y_{i-N}\det\!\left[\!\!\left(\!\!\begin{array}{l}
\displaystyle a_{0i} \vspace{0.04in}\\
\displaystyle a_{1i}\\
 .  \vspace{-0.13in} \\
 . \vspace{-0.13in} \\
 .\\
\displaystyle a_{ni} \end{array}\!\!\!\right)\!;\!\mbox{\boldmath$\rm C$}_{(n+1)\times n}\!\right] \right).\]
Taking into account that the determinant is a multilinear form (with respect to its columns) and working along the first column of the right-hand side determinants, the latter expression of $y_n$ takes the subsequent forms: 
\[\begin{array}{ll}
y_n \!\! & =(-1)^n \left(\det\!\left[\!\!\left(\!\!\begin{array}{l}
  g_0\vspace{0.04in}\\
  g_1\\
  .  \vspace{-0.13in}\\
  . \vspace{-0.13in} \\
  .\\
  g_n\end{array}\!\!\!\right)\!;\! \mbox{\boldmath$\rm C$}_{(n+1)\times n}\!\right]\!\!-\!\det\!\left[\!\left(\!\!\begin{array}{c}
\displaystyle \sum_{i=0}^{N-1}a_{0i}y_{i-N} \vspace{0.04in}\\
\displaystyle \sum_{i=0}^{N-1}a_{1i}y_{i-N}\\
  .\vspace{-0.13in} \\
 . \vspace{-0.13in} \\
 .\\
\displaystyle \sum_{i=0}^{N-1}a_{ni}y_{i-N} \end{array}\!\!\!\right) ; \mbox{\boldmath$\rm C$}_{(n+1)\times n}\right]\right)\vspace{0.1in}\\
&=(-1)^n\!\det\!\left[\!\left(\!\!\begin{array}{c}
\displaystyle g_0-\sum_{i=0}^{N-1}a_{0i}y_{i-N} \vspace{0.04in}\\
\displaystyle g_1- \sum_{i=0}^{N-1}a_{1i}y_{i-N}\\
  .\vspace{-0.13in} \\
 . \vspace{-0.13in} \\
 .\\
\displaystyle g_n- \sum_{i=0}^{N-1}a_{ni}y_{i-N} \end{array}\!\!\!\right) ; \mbox{\boldmath$\rm C$}_{(n+1)\times n}\right]\end{array}
\]
as required. \end{proof} 
As a consequence, all solutions of RO-LDEVCs, fundamental, particular and general, can be expressed as a single Hessenbergian.                   
The general solution formula in (\ref{compact general solution form}) includes the solution representation of the $N$th order linear difference equation with variable coefficients established in ~\cite{Kit:Rep} as a special case.
\section{Examples of Solutions for Irregular Order Linear Difference Equations}
\label{sec:LinearDifferenceEquationsOfIrregularOrder}
Two examples of solutions for difference equations of irregular order are discussed in this section. 
The associated row-finite matrices in these examples are of finite and infinite deficiency, possessing as fundamental solution sets a finite and a Schauder basis, respectively.
\begin{example} \label{sec:Example2} {\rm 
Consider the linear difference equation 
\begin{equation} \label{Putnam}
 (n-1)y_{n+2}-(n^2+3n-2)y_{n+1}+2n(n+1)y_n=0.
\end{equation}
}\end{example}
A closed form solution of this equation (see~\cite{Petkov}) is given by:
\begin{equation}\label{ghs2}
y_n=C_12^n+C_2n!.
\end{equation}

If $n\ge 2$, then Eq. (\ref{Putnam}) is a second order linear difference equation with polynomial coefficients. Since $(2^n)_{n\ge 2},(n!)_{n\ge 2}$ are linearly independent sequences, the fundamental theorem of linear difference equations entails that the set $\{(2^n)_{n\ge 2},(n!)_{n\ge 2}\}$ is a fundamental solution set of Eq. (\ref{Putnam}). Accordingly (\ref{ghs2}) is a closed form of the general solution of Eq. (\ref{Putnam}). 

Let us now consider the equation Eq. (\ref{Putnam}) with $n\ge 0$. The leading coefficient is $a_{n, n+2}=n-1$. Thus, Eq. (\ref{Putnam}) is of irregular order. In this case the fundamental theorem is invalid. The row-finite matrix $\A=(a_{ij})_{(i,j)\in\omega\times\omega}$ associated with Eq. (\ref{Putnam}) is of the form:
\[\A=\left(
\begin{array}{llrrrrrrrll}
 0 & 2 & -1 & 0 & 0  & 0 & 0  & 0 & 0 & 0 &...\\
 0 & 4 & -2 & 0 & 0  & 0 & 0  & 0 & 0 & 0 &...\\
 0 & 0 & 12 & -8 & 1 & 0 & 0  & 0 & 0 & 0 &...\\
 0 & 0 & 0 & 24 & -16 & 2 & 0 & 0 & 0 & 0 &...\\
 0 & 0 & 0 & 0 & 40 & -26 & 3 & 0 & 0 & 0 &...\\
 0 & 0 & 0 & 0 & 0 & 60 & -38 & 4 & 0 & 0 &...\\
 0 & 0 & 0 & 0 & 0 & 0 & 84 & -52 & 5 & 0 &...\\
 0 & 0 & 0 & 0 & 0 & 0 & 0 & 112 & -68 & 6 &...\vspace{-0.05in}\\
 . & . & . & . & . & . & . &  .  &  .  & . &...\vspace{-0.1in}\\
 . & . & . & . & . & . & . &  .  &  .  & . &...\vspace{-0.1in}\\
 . & . & . & . & . & . & . &  .  &  .  & . &...
\end{array}
\right)\]
The QHF constructed by the infinite Gauss-Jordan elimination algorithm is:
\[\QHF=\left(\begin{array}{lrrrlllllll}
0 & -2 & 1 & 0 & 0 & 0 & 0 & 0 & 0 & 0 & ...\\
0 & 0 & 0 & 0 & 0 & 0 & 0 & 0 & 0 & 0 & ...\\
0 & 24 & 0 & -8 & 1 & 0 & 0 & 0 & 0 & 0 & ...\\
 0 & 192 & 0 & -52 & 0 & 1 & 0 & 0 & 0 & 0 & ...\\
 0 & 1344 & 0 & -344 & 0 & 0 & 1 & 0 & 0 & 0 & ...\\
 0 & 9888 & 0 & -2488 & 0 & 0 & 0 & 1 & 0 & 0 & ...\\
 0 & 80256 & 0 & -20096 & 0 & 0 & 0 & 0 & 1 & 0 & ...\\
 0 & 724992 & 0 & -181312 & 0 & 0 & 0 & 0 & 0 & 1 & ...\vspace{-0.05in}\\
 . & .     & .  & .       & . & . & . &  .& . & . &...\vspace{-0.1in}\\
 . & . & . & . & . & . & . &  .  &  .  & . &...\vspace{-0.1in}\\
 . & . & . & . & . & . & . &  .  &  .  & . &...
\end{array}
\right)\]
The set of inaccessible row-lengths of $\RS(\A)$ is $\clengths_{\A}=\{0,1,3\}$. By virtue of (\ref{homogeneous solution 2}) the general solution $(y_n)_{n\in\naturals}$ of Eq. (\ref{Putnam}) is
\begin{equation}\label{ghs3} (y_0, y_1, 2y_1, y_3, -(24y_1-8y_3), -(192y_1-52y_3),-(-1344y_1-344y_3),...)^T,
\end{equation}
with free constants $y_0,y_1,y_3$. 

According to (\ref{sequence of fundamental solutions 1}) (or directly from Eq. (\ref{ghs3})) the fundamental solution set consists of three elements, the sequences:
\[\begin{array}{rrrrrrrl}
\mbox{\boldmath$\xi$}^{(0)}=(1, & 0, &  0,  & 0,   &   0,  &  0,   &   0,  & ...)^T\\
\mbox{\boldmath$\xi$}^{(1)}=(0, & 1, &  2,  & 0,    & -24,  & -192, & -1344, & ...)^T\\
\mbox{\boldmath$\xi$}^{(3)}=(0, & 0, &  0,  & 1,   &   8,  &  52,  &  344, & ...)^T.
\end{array}\]
\\Hence, if $n\ge 0$, the dimension of the space of homogeneous solutions of Eq. (\ref{Putnam}) is $\dim (\RNS (\A))=3$.

Call $\zeta=(C_12^n+C_2n!)^T_{n\ge 0}$. 
Notice that $\zeta$ also solves Eq. (\ref{Putnam}) with $n\ge 0$, but, on account of 
\[\dim(\Span(\{(2^n)^T_{n\ge 0},(n!)^T_{n\ge 0}\}))=2<\dim (\RNS (\A)),\]
$\zeta$ is not the general solution any-more.

On the other hand, as $\zeta\in\RNS (\A)$, formula (\ref{GGHS}) gives \[\zeta=\zeta_0\mbox{\boldmath$\xi$}^{(0)}+\zeta_1\mbox{\boldmath$\xi$}^{(1)}+\zeta_3\mbox{\boldmath$\xi$}^{(3)},\] where 
\[\zeta_0=C_12^0+C_20!,\ \zeta_1=C_12^1+C_21!, \ \zeta_3=C_12^3+C_23!\]
As a verification:
\[\begin{array}{ccl}  \zeta_0&=&\zeta_0\cdot 1+\zeta_1\cdot 0+\zeta_3\cdot 0 \vspace{0.05in}\\
\zeta_1&=&\zeta_0\cdot 0+\zeta_1\cdot 1+\zeta_3\cdot 0 \vspace{0.05in}\\
\zeta_2&=&C_12^2+C_22!=2(C_12^1+C_21!)=2\zeta_1=\zeta_0\cdot 0+\zeta_1\cdot 2+\zeta_3\cdot 0 \vspace{0.05in}\\
\zeta_3&=&\zeta_0\cdot 0+\zeta_1\cdot 0+\zeta_3\cdot 1 \vspace{0.05in}\\
\zeta_4&=&C_12^4+C_24!=16C_1+24C_2=(64-48)C_1+(48-24)C_2\\
&=&-48C_1-24C_2+64C_1+48C_2 \ \
=-24(2C_1+C_2)+8(8C_1+6C_2)\\
&=&-24\zeta_1+8\zeta_3\ \
=\zeta_0\cdot 0+\zeta_1\cdot (-24)+\zeta_3\cdot 8 \vspace{0.05in}\\
\zeta_5&=&C_12^5+C_25!=32C_1+120C_2=(416-384)C_1+(312-192)C_2\\
&=&-384C_1-192C_2+416C_1+312C_2\ \
 =-192(2C_1+C_2)+52(8C_1+6C_2) \\
&=&-192\,\zeta_1+52\,\zeta_3\ \
=\zeta_0\cdot 0+\zeta_1\cdot (-192)+\zeta_3\cdot 52 \vspace{-0.05in}\\
.&.&.\hspace{1in}.\hspace{1in}.\hspace{1in}.\hspace{1in}.\vspace{-0.05in}\\
.&.&.\hspace{1in}.\hspace{1in}.\hspace{1in}.\hspace{1in}.\vspace{-0.05in}\\
.&.&.\hspace{1in}.\hspace{1in}.\hspace{1in}.\hspace{1in}.
\end{array}\]
\begin{example} \label{sec:Example3} {\rm
The bivariate function 
$a_{k,m}=1-\cos \frac{(2k-m)\pi}{2}$
generates the variable coefficients of the linear difference equation:
\begin{equation} \label{example2} 
      a_{n,n+2}y_n+a_{n,n+1}y_{n-1}+...+a_{n,1}y_{-1}+a_{n,0}y_{-2}=0,\ \ n \in\naturals.
\end{equation}
Taking into account that $a_{n,n+2}=1-\cos \frac{(n-2)\pi}{2}$, the solutions in $\naturals$ of the trigonometric equation $a_{n,n+2}=0$ are given by: $n=4k+2,\ k=0,1,2,...$. Thus Eq. (\ref{example2}) is of irregular order and of infinite deficiency.}\end{example}
The row-finite matrix $\A=(a_{ij})_{(i,j)\in\omega\times\omega}$ associated with Eq. (\ref{example2}) is of the form: \[\A=\left(\begin{array}{lllllllllllllll}
 0 & 1 & 2 & 0 & 0 & 0 & 0 & 0 & 0 & 0 & 0 & 0 & 0 & 0 & ...\\
 2 & 1 & 0 & 1 & 0 & 0 & 0 & 0 & 0 & 0 & 0 & 0 & 0 & 0 & ...\\
 0 & 1 & 2 & 1 & 0 & 0 & 0 & 0 & 0 & 0 & 0 & 0 & 0 & 0 & ...\\
 2 & 1 & 0 & 1 & 2 & 1 & 0 & 0 & 0 & 0 & 0 & 0 & 0 & 0 & ...\\
 0 & 1 & 2 & 1 & 0 & 1 & 2 & 0 & 0 & 0 & 0 & 0 & 0 & 0 & ...\\
 2 & 1 & 0 & 1 & 2 & 1 & 0 & 1 & 0 & 0 & 0 & 0 & 0 & 0 & ...\\
 0 & 1 & 2 & 1 & 0 & 1 & 2 & 1 & 0 & 0 & 0 & 0 & 0 & 0 & ...\\
 2 & 1 & 0 & 1 & 2 & 1 & 0 & 1 & 2 & 1 & 0 & 0 & 0 & 0 & ...\\
 0 & 1 & 2 & 1 & 0 & 1 & 2 & 1 & 0 & 1 & 2 & 0 & 0 & 0 & ...\\
 2 & 1 & 0 & 1 & 2 & 1 & 0 & 1 & 2 & 1 & 0 & 1 & 0 & 0 & ...\\
 0 & 1 & 2 & 1 & 0 & 1 & 2 & 1 & 0 & 1 & 2 & 1 & 0 & 0 & ...\\
 2 & 1 & 0 & 1 & 2 & 1 & 0 & 1 & 2 & 1 & 0 & 1 & 2 & 1 & ...\vspace{-0.05in}\\
 . & . & . & . & . & . & . & . & . & . & . & . & . & . & ...\vspace{-0.1in}\\
 . & . & . & . & . & . & . & . & . & . & . & . & . & . & ...\vspace{-0.1in}\\
 . & . & . & . & . & . & . & . & . & . & . & . & . & . & ...
\end{array}\right).\]
The infinite Gauss-Jordan algorithm applied to $\A$ gives the sequence:
\[{\mathcal{H}}^{(0)}=(0,\ \frac{1}{2}, \ 1),  \ {\mathcal{H}}^{(1)}=\left(\begin{array}{llll} 0 & \frac{1}{2} & 1 & 0\\
 2 &      1        & 0 & 1\end{array}\right),\ {\mathcal{H}}^{(2)}=\left(\begin{array}{rlll} 2 &  1   & 0 & 0\\
-1 &  0   & 1 & 0\\
 0 &  0   & 0 & 1                                                     \end{array}\right)...
\]
After some point (here shown in the third term) the chain of submatrices of a QHF of $\A$ is constructed:
\[{\mathcal{H}}^{(2)}\!\!\mid_0=(2,\ 1), \ {\mathcal{H}}^{(2)}\!\!\mid_1=\left(\begin{array}{rlll} 2 &  1  & 0 \\
                                                                          -1 &  0   & 1 \\\end{array}\right), \ {\mathcal{H}}^{(2)}\!\!\mid_2=\left(\begin{array}{rlll} 2 &  1   & 0 & 0\\
                                             -1 &  0   & 1 & 0\\
                                              0 &  0   & 0 & 1                                                     \end{array}\right)...
\]
Following the notation of theorem \ref{main}, we conclude that $\delta_0=\delta_1=\delta_2=2, \  \delta_i=i,\ i\ge 2$. The QHF constructed by the infinite Gauss-Jordan elimination algorithm is of the form:
\[\QHF=\left(
\begin{array}{rlllrlllrllllll}
 2 & 1 & 0 & 0 & 0 & 0 & 0 & 0 & 0 & 0 & 0 & 0 & 0 & 0 & ...\\
 -1 & 0 & 1 & 0 & 0 & 0 & 0 & 0 & 0 & 0 & 0 & 0 & 0 & 0 & ...\\
 0 & 0 & 0 & 1 & 0 & 0 & 0 & 0 & 0 & 0 & 0 & 0 & 0 & 0 & ...\\
 0 & 0 & 0 & 0 & 2 & 1 & 0 & 0 & 0 & 0 & 0 & 0 & 0 & 0 & ...\\
 0 & 0 & 0 & 0 & -1 & 0 & 1 & 0 & 0 & 0 & 0 & 0 & 0 & 0 & ...\\
 0 & 0 & 0 & 0 & 0 & 0 & 0 & 1 & 0 & 0 & 0 & 0 & 0 & 0 & ...\\
 0 & 0 & 0 & 0 & 0 & 0 & 0 & 0 & 0 & 0 & 0 & 0 & 0 & 0 & ...\\
 0 & 0 & 0 & 0 & 0 & 0 & 0 & 0 & 2 & 1 & 0 & 0 & 0 & 0 & ...\\
 0 & 0 & 0 & 0 & 0 & 0 & 0 & 0 & -1 & 0 & 1 & 0 & 0 & 0 & ...\\
 0 & 0 & 0 & 0 & 0 & 0 & 0 & 0 & 0 & 0 & 0 & 1 & 0 & 0 & ...\\ 
 0 & 0 & 0 & 0 & 0 & 0 & 0 & 0 & 0 & 0 & 0 & 0 & 0 & 0 & ...\\
 0 & 0 & 0 & 0 & 0 & 0 & 0 & 0 & 0 & 0 & 0 & 0 & 2 & 1 & ...\vspace{-0.05in}\\
 . & . & . & . & . & . & . & . & . & . & . & . & . & . & ...\vspace{-0.1in}\\
 . & . & . & . & . & . & . & . & . & . & . & . & . & . & ...\vspace{-0.1in}\\
 . & . & . & . & . & . & . & . & . & . & . & . & . & . & ...
\end{array}
\right).
\]
By applying to the identity $\omega\times\omega$ matrix the same sequence of row elementary operations applied by the Gauss-Jordan elimination to the initial matrix $\A$, we obtain the nonsingular matrix:
\[\Q=\left(
\begin{array}{rrrrrrrrrrrll}
 1 & 1 & -1 & 0 & 0 & 0 & 0 & 0 & 0 & 0 & 0 & 0 & ...\\
 0 & -\frac{1}{2} & \frac{1}{2} & 0 & 0 & 0 & 0 & 0 & 0 & 0 & 0
   & 0 & ...\\
 -1 & 0 & 1 & 0 & 0 & 0 & 0 & 0 & 0 & 0 & 0 & 0 & ...\\
 0 & -1 & 0 & 1 & 0 & 0 & 0 & 0 & 0 & 0 & 0 & 0 & ...\\
 0 & \frac{1}{2} & -\frac{1}{2} & -\frac{1}{2} & \frac{1}{2} &
   0 & 0 & 0 & 0 & 0 & 0 & 0 & ...\\
 0 & 0 & 0 & -1 & 0 & 1 & 0 & 0 & 0 & 0 & 0 & 0 & ...\\
   0 & 0 & 0 & 1 & -1 & -1 & 1 & 0 & 0 & 0 & 0 & 0 & ...\\
 0 & 0 & 0 & 0 & 0 & -1 & 0 & 1 & 0 & 0 & 0 & 0 & ...\\
 0 & 0 & 0 & \frac{1}{2} & -\frac{1}{2} & 0 & 0 & -\frac{1}{2}
   & \frac{1}{2} & 0 & 0 & 0 & ...\\
 0 & 0 & 0 & 0 & 0 & 0 & 0 & -1 & 0 & 1 & 0 & 0 & ...\\
 0 & 0 & 0 & 0 & 0 & 0 & 0 & 1 & -1 & -1 & 1 & 0 & ...\\
 0 & 0 & 0 & 0 & 0 & 0 & 0 & 0 & 0 & -1 & 0 & 1 & ...\vspace{-0.05in}\\
 . & . & . & . & . & . & . & . & . & .  & . & . & ...\vspace{-0.1in}\\
 . & . & . & . & . & . & . & . & . & . & . & .  & ...\vspace{-0.1in}\\
 . & . & . & . & . & . & . & . & . & . & . & .  & ...
\end{array}
\right).\]
As expected, $\Q\cdot \A=\QHF$, thus preserving row equivalence. \\The indexing set of zero rows of $\QHF$ is $W=\{6+4k, k\in \naturals\}$ and so a Hamel basis of the left-null space of $\A$ is
$\{\e_{6+4k}-\e_{5+4k}-\e_{4+4k}+\e_{3+4k}\}_{k\in \naturals}$.
\\In view of $\QHF$, the set of inaccessible row lengths of $\RS(\A)$ is 
$\clengths_{\A}=\{4n,\ n\in\naturals\}$.
\\As $\defic(\A)=\defic(\QHF)=\card(\clengths_{\A})=\infty$, it follows that the fundamental solution set $\mbox{\boldmath$\xi$}=\{\mbox{\boldmath$\xi$}^{(s)}\}_{s\in \clengths_{\A}}$ of Eq. (\ref{example2}) is a Schauder basis of the space of its homogeneous solutions. Thus Eq. (\ref{sequence of fundamental solutions 1}) gives
\[\begin{array} {llccc}\!\!\!\!\!\!
\mbox{\boldmath$\xi$}^{(s)}=(0,0,...,0,\!\!&1, \!\!\!\!\! &\!\!\!\!\!\! -2, \!\!\!\!\!  &\!\!\!\! 1,     \!\!\!\!\!\! &\!\!\!\!\!  0,...,0,...)^T, \ \ s=0,4,8,12,...  \\ 
& \uparrow     &\uparrow         &\uparrow  \\
& s            &\!\!\! s\!+\!1      \!\!\!\!\!  & s\!+\!2        
\end{array} \] 
or $\mbox{\boldmath$\xi$}=\{\mbox{\boldmath$\xi$}^{(4n)},\ n\in\naturals\}$ with
\[\mbox{\boldmath$\xi$}^{(4n)}=\e_{4n}-2\e_{4n+1}+\e_{4n+2}.\]
By virtue of (\ref{GGHS}), the general solution of Eq. (\ref{example2}) is given by
\[y=\sum_{n=0}^{\infty}c_{n}(\e_{4n}-2\e_{4n+1}+\e_{4n+2}),\]
where $c_{n}, n=0,1,2,...$ are free constants.

\end{document}